\documentclass[11pt,oneside,letter]{article}
\usepackage{amsthm, amsmath, amssymb, amsfonts, graphicx, epsfig, bbm}
\usepackage{mathrsfs}
\usepackage{enumerate}
\usepackage[colorlinks=true, pdfstartview=FitV, linkcolor=blue,
            citecolor=blue, urlcolor=blue]{hyperref}
\usepackage[usenames]{color}
\usepackage[numbers]{natbib}
\usepackage{url}
\makeatletter\def\url@leostyle{%
 \@ifundefined{selectfont}{\def\UrlFont{\sf}}{\def\UrlFont{\scriptsize\ttfamily}}} \makeatother
\usepackage{accents, comment}

\usepackage[normalem]{ulem}

\newtheorem{theorem}{Theorem}
\newtheorem{proposition}[theorem]{Proposition}
\newtheorem{lemma}[theorem]{Lemma}
\newtheorem{corollary}[theorem]{Corollary}
\newtheorem{definition}[theorem]{Definition}

\theoremstyle{definition}

\newtheorem{example}[theorem]{Example}
\theoremstyle{remark}
\newtheorem{remark}[theorem]{Remark}

\numberwithin{equation}{section}
\numberwithin{theorem}{section}
\def\cA{\mathcal{A}}

\def\cF{\mathcal{F}}

\def\cY{\mathcal{Y}}

\def\bE{\mathbb{E}}

\def\bP{\mathbb{P}}
\def\bQ{\mathbb{Q}}

\def\bse{\begin{equation}\begin{split}}
\def\ese{\end{split}\end{equation}}

\newcommand{\1}{\mathbbm{1}}                     % preferable way of writing indicator function
\newcommand{\set}[1]{{\{#1\}}}            % set: {xyz} to be used for inline formulas
\newcommand{\Set}[1]{\left\{#1\right\}} % set: {xyz} to be used for seapare (not inline) formulas
\renewcommand{\mid}{\;|\;}              % mid bar with small spaces before and after: x | y
\newcommand{\abs}[1]{\left\vert#1\right\vert}   % absolute value
\usepackage{cancel}
\usepackage[T1]{fontenc}
\usepackage[utf8]{inputenc}
\usepackage{varwidth}
\usepackage{upgreek}
\title{ Conditional Markov Chains\\ Part II: Consistency and Copulae }
\author{Tomasz R.  Bielecki$^1$ \\[-0.3ex]
\url{bielecki@iit.edu} \\[-0.3ex]
\and
Jacek Jakubowski$^{2,3}$ \\[-0.3ex]
\url{jakub@mimuw.edu.pl} \\[-0.3ex]
\and
Mariusz Niew\k{e}g\l owski$^3$
\\[-0.3ex]
\url{m.nieweglowski@mini.pw.edu.pl} \\[-0.3ex]
\and \small{$^1$Department of Applied Mathematics,}\\[-0.3ex]
\small{Illinois Institute of Technology,}\\[-0.3ex]
\small{Chicago, IL 60616, USA }\\[-0.3ex]
\and
\small{$^2$Institute of Mathematics,}\\[-0.3ex]
\small{University of Warsaw,}\\[-0.3ex]
\small{Banacha 2,
02-097 Warszawa, Poland}\\[-0.3ex]
\and
\\[-0.3ex]
\small{$^3$Faculty of Mathematics
and Information Science,}\\[-0.3ex]
\small{Warsaw University of Technology,}\\[-0.3ex]
\small{ul. Koszykowa 75,
00-662 Warszawa, Poland}\\[-0.3ex]
}

\date{\today}

\binoppenalty=\maxdimen %no breaks inside math formulas..
\relpenalty=\maxdimen

\usepackage{dsfont}
\newcommand{\be}{\begin{equation}}
\newcommand{\ee}{\end{equation}}
\newcommand{\bde}{\begin{displaymath}}
\newcommand{\ede}{\end{displaymath}}
\newcommand{\beq}{\begin{eqnarray*}}
\newcommand{\eeq}{\end{eqnarray*}}
\newcommand{\beqa}{\begin{eqnarray}}
\newcommand{\eeqa}{\end{eqnarray}}
\newcommand{\bel }{\left\{\begin{array}{ll}}
\newcommand{\eel}{\cr \end{array} \right.}
\newcommand{\bd}{\begin{definition} \rm }
\newcommand{\ed}{\end{definition} \rm }
 \newcommand{\bex}{\begin{example} \rm }
\newcommand{\eex}{\end{example}}
\newcommand{\bt}{\begin{theorem}}
\newcommand{\et}{\end{theorem}}
\newcommand{\bl}{\begin{lemma}}
\newcommand{\el}{\end{lemma}}
\newcommand{\bp}{\begin{proposition}}
\newcommand{\ep}{\end{proposition}}
\newcommand{\bcor}{\begin{corollary}}
\newcommand{\ecor}{\end{corollary}}
\newcommand{\lab }{\label }
\newcommand{\brem}{\begin{remark}}
\newcommand{\erem}{\end{remark}}

\def\proof{\noindent {\it Proof. $\, $}}
\def\finproof{\hfill $\Box$ \vskip 5 pt}

\def\I{\mathds{1}}

\def \wh{\widehat}
\def \wt{\widetilde}
\def\r{\mathbb R}
\def\F{{\mathcal F}}

\def\G{{\mathcal G}}

\def\FF{{\mathbb F}}

\def\GG{{\mathbb G}}
\def\II{{\mathbb I}}

\def\P{\mathbb P}
 \def\Q{\mathbb Q}

\def\EP{{\mathbb E}_{\mathbb P}}

\def\E {{\mathbb E} }
\def\tT{{t\in[0,T]}}

\DeclareMathAlphabet\mathbfcal{OMS}{cmsy}{b}{n}

 \mathchardef\mhyphen="2D

    \usepackage{fancyhdr}
    \usepackage{lastpage}
\begin{document}
\maketitle
\begin{abstract}
In this paper we continue the study of conditional Markov chains (CMCs) with finite state spaces, that we initiated in Bielecki, Jakubowski and Nieweglowski (2015). Here, we turn our attention to the study of Markov consistency and Markov copulae with regard to CMCs, and thus we follow up on the study of Markov consistency and Markov copulae for ordinary Markov chains that we presented in Bielecki, Jakubowski and Nieweglowski (2013).
\\
{\noindent \small
  }
\end{abstract}
\tableofcontents

%%%%%%%%%%%%%%%%%%%%%%%%%%%%%%%%%%%%%%%%%%%%%%%%%%%%%%%%%
\section{Introduction}
%%%%%%%%%%%%%%%%%%%%%%%%%%%%%%%%%%%%%%%%%%%%%%%%%%%%%%%%%
In this paper we continue the study of conditional Markov chains (CMCs) with finite state spaces, that we initiated in Bielecki, Jakubowski and Niew\k{e}g\l owski \cite{BieJakNie2015}  and continued in \cite{BieJakNie2014a}. Here, we turn our attention to the study of Markov consistency and Markov copulae with regard to CMCs,which are doubly stochastic Markov chains (DSMC) and thus we follow up on the study of Markov consistency and Markov copulae for ordinary Markov chains that we presented in Bielecki, Jakubowski and Niew\k{e}g\l owski  \cite{BieJakNie2013}.

The concepts of Markov consistency and Markov copulae were developed in the context of the problem of modeling of multivariate stochastic processes in such a way that distributional laws of the univariate components of a multivariate process agree with given, predetermined laws (cf. Bielecki, Vidozzi and Vidozzi \cite{BieVidVid2008}, Bielecki, Jakubowski,  Vidozzi and Vidozzi  \cite{BieJakVidVid2008}, Bielecki, Jakubowski and Niew\k{e}g\l owski \cite{BieJakNie2010} and \cite{BieJakNie2013}, Liang and Dong \cite{LiangDong}). So, essentially, this is a problem of modeling (non-trivial%\footnote{"Non-triviality" is of course a relative concept. For example, in case  }
) dependence between (univariate) stochastic processes with given laws.

Modeling of dependence between stochastic processes is a very important issue arising from many different  applications, among others in finance and in insurance. In \cite{BieJakNie2013} we focused on modeling dependence between ordinary Markov chains.  Specifically, we studied the general problem of constructing a multivariate Markov chain such that its components have laws identical with the laws of given Markov chains.

In this paper we elevate the study done in \cite{BieJakNie2013} to the world of conditional Markov chains. This
is done in response to the need for modeling dependence between dynamic systems in  cases when some conditional properties of a system are important and should be accounted for. We refer to Section \ref{BGW}, where we discuss a relevant practical problem.

We introduce and study the concepts of {\it strong Markovian consistency} and {\it weak Markovian
consistency} for conditional Markov chains.  Accordingly, we introduce and study the concepts of strong Markov copulae and weak Markov copulae for conditional Markov chains, which we call {\it strong CMC copulae} and {\it weak CMC copulae}, respectively. We refer to the discussion of practical relevance of the concepts of strong/weak Markov copulae that was done in  \cite{BieJakNie2013}
({see also Bielecki, Cousin, Cr\'{e}pey and Herbertsson \cite{BCCH2014} and Jakubowski and Pytel \cite{JakPyt2015}}). Much of what was said there applies in the context of strong/weak CMC copulae.

 As already said, we confine our study to the case of finite CMCs. One might object the choice of finite CMCs as the object of interest in this paper, as one might think that this choice is very restrictive. In Bielecki, Jakubowski and Niew\k{e}g\l owski \cite{BieJakNie2012} we studied strong Markovian dependence in the context of (nice) Feller processes, whereas in \cite{BieJakNie2013} we studied strong and weak Markovian dependence in the context of finite Markov chains. What we have learned from \cite{BieJakNie2013}  is that
 from the point of view of modeling of dependence between components of a multivariate Markov process, the finite state space set-up is actually not restrictive at all! Likewise, as it will be seen thought this paper, studying the concepts of consistency and copulae in case of finite CMCs is quite challenging and by no means restrictive.

 From the mathematical perspective, the problem of modeling of dependence between CMCs, generalizes the problem of modeling of dependence between random times. The latter problem is one of the key problems studied in the context of portfolio credit risk and counterparty risk, in case when one only considers two possible states of financial obligors: the pre-default state and the default state, with the additional caveat that the default state is absorbing, and the issue in question is the issue of modeling dependence between default times of various obligors (cf. e.g. Bielecki, Cousin, Cr\'{e}pey and Herbertsson \cite{BCCH2014}). The study done in this paper allows for tackling more general problems, such as the problem of modeling of dependence between evolutions of credit ratings of financial obligors in cases where conditioning reference information is relevant; in particular, it opens a door for generalizing  the set-up that was used in Biagini, Groll and  Widenmann \cite{BiaWid2013} to deal with an interesting problem of evaluation  of  premia  for unemployment insurance products for a pool of individuals.

The paper is organized as follows: In Section 2 we recall the basic set-up of the companion paper \cite{BieJakNie2014a}, which will be used in the rest of the paper. In Sections 3 and 4 we introduce and study the concepts of strong and weak consistency for CMCs, respectively. Section 5 is devoted to presentation of strong and weak CMC copulae, and related examples. In Section 6 we propose a possible application of the theory developed in this paper. Finally, in the Appendix, we collect several technical results that are used throughout the text.

%\textcolor[rgb]{0.00,0.00,1.00}{In Section 1 we give a sufficient
%and necessary condition for a multivariate Markov chain to be weakly
%consistent. Note that a sufficient condition for weak Markovian
%consistency can be deduced from the result of Rogers and Pitman
%\cite{RogPit1981} in which   sufficient conditions for  a function of a
%Markov process to be  a Markov process are given. Our condition for a
%weak Markovian consistency is not only more explicit, but also
%necessary. We also study the question when weak Markovian
%consistency implies strong Markovian consistency. It turns out that
%this is equivalent to $\P$-immersion between $\FF^{X^i}$ and
%$\FF^X$, given that weak Markovian consistency holds.
% In Section 2 we  study weak Markov copulae.
%%, characterizing them in terms of a system of algebraic equations.
%In Section 3  we present  three simple, but non-trivial examples,
%that illustrate \textbf{intricacies of dependence} between components of a
%multivariate Markov chain. Specifically, in Examples 3.1--3.3 we show that
%\begin{enumerate}
%\item there exist Markov processes that are strongly Markovian consistent,
%\item there exist Markov processes that are weakly Markovian consistent, but are not strongly Markovian consistent; in addition, in this case, one would expect that even if a multivariate Markov process is time-homogeneous, its components are time-inhomogeneous Markov processes; Example 3.2 illustrates this,
%\item there exist Markov processes that are neither strongly Markovian consistent nor weakly Markovian consistent.
%\end{enumerate}}
%
%\tr{wspomniec buczka}

\section{Preliminaries}

We begin with recalling the basic set-up and the basic definitions from \cite{BieJakNie2014a}.

Let $T>0$ be a fixed finite time horizon. Let $(\Omega, \mathcal{A}, \mathbb{P})$ be an underlying complete probability space, which is endowed with two filtrations, the reference filtration $\FF=(\F_t)_{t\in[ 0,T]}$ and another filtration\footnote{The reference filtration will be kept unchanged throughout the paper, whereas the meaning of filtration $\GG$ will change depending on the context.} $\GG=(\G_t)_{t\in[ 0,T]}$, that are assumed to satisfy the usual conditions.
Typically, processes considered in this paper are defined on  $(\Omega, \mathcal{A}, \mathbb{P})$, and are restricted to the time interval $[0,T]$.
{ Moreover, for any process $U$
we denote  by $\FF^U$ the completed right-continuous filtration generated by this process.}

 In addition, we fix a finite set $S$, and we denote by $d$ the cardinality of $S$. Without loss of generality we take $S=\set{1,2,3,\ldots,d}.$

An $S$-valued,  $\mathbb{G}$-adapted c\` adl\` ag  process $X$ is called  an $(\mathbb{F},\mathbb{G})$--conditional Markov chain if for every $x_1, \ldots, x_k \in S$ and  for every $0\leq t\leq t_1 \leq \ldots \leq t_k\leq T$ it satisfies the following property$\, $
\begin{equation}\label{eq:CMC-def-1}
\P( X_{t_k} =x_k, \ldots, X_{t_1} = x_1 | \F_t \vee \G_t)
=
\P( X_{t_k} =x_k, \ldots, X_{t_1} = x_1 | \F_t \vee \sigma(X_t)).
\end{equation}
As in \cite{BieJakNie2014a} we write  $(\FF, \GG)$-CMC, for short, in place of  $(\FF, \GG)$-conditional Markov chain.

Given an $(\FF,\GG)$-CMC process $X$, we define its indicator process,
\begin{equation}\label{eq:def-Hi}
H^x_t := \I_\set{ X_t = x },\quad x \in S,  \quad t\in[0,T].
\end{equation}
Accordingly, we define a column vector  $ H_t =(H^x_t,\  x\in S )^\top  $, where $\top$
denotes transposition. Similarly, for $x,y\in S,\ x\ne y,$ we define process $H^{xy}$  that counts the number of transitions of $X$ from $x$ to $y$,
\begin{equation}\label{eq:def-Hxy}
	H^{xy}_t := \# \set{ u
\leq t: X_{u-} = x\  \textrm{and}\ X_u =y } = \int_{] 0,t]} \!\!\!H^{x}_{u-} d H^{y}_u, \quad \tT.
\ee

We say that an $\FF$-adapted (matrix valued) process $\Lambda_t=[\lambda^{xy}_t]_{x,y \in S }$ satisfying
\begin{align}\label{eq:int-cond}
	\lambda^{xy}_t \geq 0, \quad \forall x, y \in S, x \neq y, \quad and \quad
	\sum_{y \in S }\lambda^{xy}_t = 0, \quad \forall x \in S,
\end{align}
is an $\FF$-intensity matrix process for $X$, if the process $M =(M^x_t,\  x\in S )^\top$ defined as
\be\label{eq:Mart-M-nowy}
     M_t = H_t - \int_0^t \Lambda_u^\top H_{u} du, \quad t\in[0,T],
\ee
is an $\FF \vee \GG\,$--$\,$local martingale (with values in $\r^d$).

Similarly as in \cite{BieJakNie2014a}, we impose in the present work the following restriction:
\begin{center}
 \framebox[1.02\width][c]{ \strut In the rest of this paper we restrict ourselves to CMCs, which admit $\FF$-intensity.}
\end{center}

Most of the analysis done in \cite{BieJakNie2014a} regards $(\FF, \GG)$--CMCs that are also $(\FF, \GG)$ doubly stochastic Markov chains. This is because doubly stochastic Markov chains enjoy some very useful analytical properties.
We will now recall the concept of  $(\mathbb{F},\mathbb{G})$-doubly stochastic Markov chain, $(\FF, \GG)$--DSMC for short, that was introduced in  Jakubowski and Niew\k{e}g{\l}owski \cite{JakNie2010}. A  $\mathbb{G}$-adapted c\` adl\` ag  process $X=(X_t)_{t \in [0,T]}$ is called  an $(\mathbb{F},\mathbb{G})$--DSMC with state space $S$  if
for any $0 \leq s \leq t \leq T $ and
    every $y \in S$ we have
 \be\lab{eq:DSMC-def}
 \mathbb{P}( X_t = y \mid \mathcal{F}_T \vee \mathcal{G}_s )
 =    \mathbb{P}( X_t = y \mid \mathcal{F}_t \vee \sigma(X_s) ).
 \ee
 We recall from \cite{BieJakNie2014a} that  with any $X$, which is an $(\FF,\GG)$-DSMC,  we  associate a matrix valued random field $ P=( P(s,t),\  0 \leq s \leq t \leq T)$, called the conditional transition
probability matrix field (c--transition field for short), where $ P(s,t)= ( p_{xy}(s,t))_{x,y \in
S }$   is defined by
\be\label{tp}   p_{x,y}(s,t) =
    \frac{ \mathbb{P}( X_t = y , X_s = x \mid \mathcal{F}_t  ) }{  \mathbb{P}(  X_s = x  | \F_t ) }
\1_\set{\mathbb{P}(  X_s = x  | \F_t ) > 0}
+
\I_\set{ x=y} \1_\set{\mathbb{P}(  X_s = x  | \F_t ) = 0}
.
\ee
By \cite[Proposition 4.2]{BieJakNie2014a} we know that for any $0 \leq s \leq t \leq T $ and  for
    every $y \in S$ we have
    \begin{equation}\label{eq:DSMC-def-3}
    \mathbb{P}( X_t = y \mid \mathcal{F}_T \vee \mathcal{G}_s )
     = \sum_{x \in S}\I_\set{ X_s  = x }
    p_{xy}(s,t).
    \end{equation}
Moreover we recall that
% \begin{definition}\label{defdef}
%We say that
$\mathbb{F}$--adapted matrix-valued
process $\Gamma = (\Gamma_s)_{s \geq 0} = ([\gamma^{xy}_s]_{x,y \in S})_{s
\geq 0}$ is an intensity of an   $(\mathbb{F},\mathbb{G})$-DSMC $X$
if:
%\footnote{The symbol $"\mathrm{I}``$  used below is a generic symbol for the identity matrix, whose dimension may vary depending on the context.} \\

\noindent 1)
\begin{equation}  \label{nr2/8}
\int_{]0,T]} \sum_{x \in S} \abs{\gamma^{xx}_s } ds <
\infty  .\ \ \ \ \ \
\end{equation}
 2)
  \begin{equation}\label{eq:intensity-cond}
   \gamma^{xy}_s \geq 0 \ \ \
    \forall x, y\in S, x\neq y
    ,
    \ \ \ \
    \gamma^{xx}_s = - \sum_{y \in S: y \neq x} \gamma^{xy}_s
    \ \ \ \
    \forall x \in  S.
\end{equation}
3)
  The Kolmogorov backward equation holds: for all  $v\leq t$,
\begin{equation}\label{eq:INT-trans-prob-backward}
 P(v,t) - \mathrm{I}     =   \int_v^t  \Gamma_u P(u,t) du.
\end{equation}
4)  The Kolmogorov forward equation holds: for all  $v\leq t$,
\begin{equation}\label{eq:INT-trans-prob-forward}
 P(v,t)  - \mathrm{I} = \int_v^t P(v,u) \Gamma_u du.
\end{equation}
%\end{definition}

We refer to \cite{BieJakNie2014a} for discussion of the notion of intensity process of an  $(\FF, \GG)$--DSMC, as well as for a discussion of the relationship between the concept of the $(\FF, \GG)$--CMC and the concept of $(\FF, \GG)$--DSMC. In particular, it is shown in \cite{BieJakNie2014a} that one can construct an  $(\FF, \GG)$--CMC, which is also an $(\FF, \GG)$--DSMC. In addition, sufficient conditions under which an $(\FF, \GG)$--DSMC is an $(\FF, \GG)$--CMC are given in \cite{BieJakNie2014a}.

In what follows, we will use the acronym $(\FF, \GG)$--CDMC for any process that is both an $(\FF, \GG)$--CMC and an $(\FF, \GG)$--DSMC
.

\section{Strong Markovian Consistency of Conditional Markov Chains}
\label{sec:SMC}

Let $X=(X^1,\ldots,X^N)$ be a multivariate ${(\FF,\FF^X)}$-CMC\,\footnote{{Definitions of strong and weak Markov consistency can be naturally extended to the case of process $X=(X^1,\ldots,X^N)$, which is a multivariate ${(\FF,\GG)}$-CMC, with $\FF^X\subseteq \GG$. In the present paper we shall only work with $X=(X^1,\ldots,X^N)$ being  a multivariate ${(\FF,\FF^X)}$-CMC.}} with values in $S:=\text{\sf X}_{k=1}^N\, S_k$, where $S_k$ is a finite set, $k=1,\ldots,N.$ We will denote by  $\Lambda$ the $\FF$-intensity of $X$.

\begin{definition}
(i) Let us fix $k\in\set{1,\ldots,N}$.  We say that  process $X$ satisfies the
 strong Markovian consistency property with respect to $(X^k, \mathbb{F})$ if  for every $x^k_1,  \ldots, x^k_m\in S_k $
and for all $0 \leq t \leq t_1 \leq \ldots \leq t_m \leq T,$ it holds that
\begin{equation}\label{eq:markov-cons-strong}
\mathbb{P}\left( X^k_{t_m } =x^k_{m},  \ldots, X^k_{t_1 } =x^k_{1}| \mathcal{F}_t \vee  \mathcal{F}^X_t \right) =   \mathbb{P} \left( X^k_{t_m } =x^k_{m},  \ldots, X^k_{t_1 } =x^k_{1}  | \mathcal{F}_t \vee \sigma(X^k_t)    \right),
\end{equation}
or, equivalently, if  $X^k$ is an $(\mathbb{F},  \FF^X)$-CMC.\footnote{In more generality, one might define strong Markovian consistency with respect to a collection $X^I:=\{X^k,k\in I\subset \{1,2,\ldots\}\}$ of components of $X$. This will not be done in this paper though.} \\
(ii) If   $X$ satisfies the
 strong Markovian consistency property with respect to $(X^k, \mathbb{F})$ for all $k\in\set{1,\ldots,N}$, then we say that  $X$ satisfies the
 strong Markovian consistency property with respect to $\mathbb{F}$.
\end{definition}
\begin{remark}
There is a relation between strong Markovian consistency of $X$ with respect to $(X^k, \mathbb{F})$ and the concept of Clive Granger's causality (cf. Granger \cite{Granger}):  Suppose that process $X$ satisfies the
 strong Markovian consistency property with respect to $(X^k, \mathbb{F})$. If the reference filtration $\mathbb{F}$ is trivial, then the collection $\{X^i,\ i\ne k\}$ does not Granger cause $X^k$. By extenso, we may say that, in the case when reference filtration $\mathbb{F}$ is not trivial, then, ''conditionally on $\mathbb{F}$``, the collection $\{X^i,\ i\ne k\}$ does not Granger cause $X^k$.
\end{remark}

The next definition extends the previous one by requiring that the laws of the marginal processes $X^k,\ k=1,\dots , N,$ are predetermined. This definition will be a gateway to the concept of strong CMC copulae in Section \ref{SCMCC}.
 \begin{definition}
Let $\cY=\set{Y^1, \ldots, Y^N}$ be a family of processes such that each $Y^k $ is an $(\mathbb{F},\mathbb{F}^{Y^k})$-CMC with values in $S_k$.\\
(i) Let us fix $k\in \set{1,2,\ldots,N}$ and let process $X$ satisfy the strong Markovian consistency property with respect to  $(X^k, \mathbb{F})$.  If the conditional law of $X^k$ given $\mathcal{F}_T$ coincides with the conditional law of $Y^k$ given $\mathcal{F}_T$, then we say that  process $X$ satisfies the strong Markovian consistency property with respect to  $(X^k,\mathbb{F}, Y^k)$.\\
(ii)  If $X$ satisfies the strong Markovian consistency property with respect to  $(X^k,\mathbb{F}, Y^k)$  for every $k\in \set{1,2,\ldots,N}$, then we say that $X$ satisfies the strong Markovian consistency property with respect to  $(\mathbb{F},\cY)$.
\end{definition}

\subsection{Sufficient and Necessary Conditions for Strong Markovian Consistency}\label{sekcjazrameczka}

In view of the methodology developed in \cite{BieJakNie2014a}, specifically, in view of \cite[Proposition 4.17]{BieJakNie2014a} %\ref{thm:CMC-DCMC}
 therein, we will frequently make use in what follows of the following assumption

\begin{center}
\framebox[1.02\width][c]{ \strut
{\sf Assumption (A):}
\begin{tabular}{l}
(i)
 $X$ is an $(\FF,\FF^X)$--DSMC admitting an
intensity. \\
(ii) $\mathbb{P}( X_{0} = x_0 | \F_T) =\mathbb{P}( X_{0} = x_0 | \F_0)$
 for every $x_0 \in S$. \\
\end{tabular}
}
\end{center}

%\begin{center}
%{\framebox[1.02\width][c]{
%{{\sf Assumption (A):} \begin{itemize}
%                         \item[(i)] $X$ is an $(\FF,\FF^X)$--DSMC {admitting an
%intensity.}
%                         \item[(ii)] $\mathbb{P}( X_{0} = x_0 | \F_T) =\mathbb{P}( X_{0} = x_0 | \F_0) \quad
%\text{for every $x_0 \in S$,}$}
%                       \end{itemize}}
%}
%\end{center}

Let us note that in view of \cite[Theorem 4.15]{BieJakNie2014a}, under Assumption (A)(i), the intensity of $X$ considered as  an $(\FF,\FF^X)$--DSMC coincides, in the sense of \cite[Definition 2.5]{BieJakNie2014a}, with the $\FF$-intensity $\Lambda$ of $X$ considered as  an $(\FF,\FF^X)$--CMC. Consequently, we will now say that $X$ is an $(\FF,\FF^X)$--CDMC with intensity $\Lambda.$

%\mn{Jezeli nie udowodnimy punktu ii) Propozycji 4.17 w \cite{BieJakNie2014a} to musi zostac intenisty a nie $\FF$-intensity. Wtedy trzeba by tez zdefiniowac co to jest intensity dla DSMC }
%\vskip 10 pt

Our next goal is to provide condition characterizing strong Markovian consistency of process $X$. Towards this end we first introduce the following condition\footnote{The acronym SM comes from  Strong Markov.}
\vskip 10pt
\noindent
{\sf Condition (SM-k):$\ $} There exist
 $\FF$-{adapted} processes $\lambda^{k;x^ky^k}$, $x^k,y^k \in S_k$, $x^k\neq y^k$, such that
\be\label{weak-Mk}
\begin{aligned}
&\I_\set{X^k_t = x^k} \sum_{\substack{y^n\in S_n,\\ n=1,2,\ldots,N, n \neq k}}\lambda^{(X^1_t,\ldots,X^{k-1}_t,x^k,X^{k+1}_t,\ldots,X^N_t)(y^1,\ldots,y^k,\ldots,y^N)}_t \\
&=
\I_\set{X^k_t = x^k} \lambda^{k;x^k y^k }_t,
\quad dt \otimes d \P\mhyphen a.e. \quad  \forall x^k, y^k \in S_k, x^k \neq y^k.
\end{aligned}
\ee

We have following proposition, which is a direct consequence of \cite[Proposition 2.6]{BieJakNie2014a}, and thus we omit its proof.
\begin{proposition}
Let $X$ satisfy  Assumption (A) and let $\Lambda$, $\wh{\Lambda}$ be $\FF$-intensities of $X$. Then (SM-k)  holds for $\Lambda$ if and only if it holds for $\wh{\Lambda}$.
\end{proposition}

\noindent The next theorem provides sufficient and necessary conditions for strong Markovian consistency property of process $X$  with respect to $\mathbb{F}$.
\bt\label{thm:weak-Mk} Let $X$ satisfy Assumption (A), and let us fix $k\in \set{1,2,\ldots,N}$. Then, $X$ is strongly Markovian consistent with respect to $(X^k, \FF) $ if and only if Condition (SM-k) is satisfied.  Moreover, in this case, process $X^k$ admits the $\FF$-intensity
$\Lambda^k=[\lambda^{k;{x^ky^k}}]_{{x^k,y^k} \in S_k},$ with $\lambda^{k;{x^ky^k}}$ as in  Condition (SM-k), and with $\lambda^{k;{x^kx^k}}$ given by
\[
\lambda^{k;{x^kx^k}}_t= -\sum_{y^k \in S_k, y^k \neq x^k}\lambda^{k;{x^ky^k}}_t.
\]
\et

\begin{proof} Let $\Lambda$ be an $\FF$ intensity of $X$. For simplicity of notation, but without loss of generality, we give the proof for  $k=1$, and for $N=2$,  so that  $S=S_1 \times S_2$,  $X =(X^1, X^2)$. In this case, \eqref{weak-Mk} takes the form
\begin{align}\label{weak-M1}
\I_\set{X^1_t = x^1} \sum_{ y^2 \in S_2} \lambda^{(x^1,X^2_t)(y^1,y^2)}_t
=
\I_\set{X^1_t = x^1} \lambda^{1;x^1 y^1}_t,
\quad dt \otimes d \P\mhyphen a.e., \quad  \forall x^1, y^1 \in S_1, x^1 \neq y^1.
\end{align}
\noindent \underline{Step 1:}

For $x^1, y^1\in S_1,$ $x^1 \neq y^1$ and for $x^2\in S_2$ we define processes $H^{1;x^1y^1}, H^{1; x^1}, H^{2;x^2}$ by
\[
	H^{1;x^1y^1}_t :=
\sum_{0 < u \leq t } \I_\set{X^1_{u-}= x^1, X^1_{u}= y^1}=\sum_{x^2, y^2 \in S_2 } H^{(x^1,x^2)(y^1,y^2)}_t, \quad t \in [0,T],
\]
and
\[
	H^{1; x^1}_t := \I_\set{ X^1_t = x^1}
,\quad 	H^{2; x^2}_t := \I_\set{ X^2_t = x^2}.
\]

Next, we consider process $K^{(x^1,x^2) (y^1,y^2)}$, given as
\begin{align*}
K^{(x^1,x^2) (y^1,y^2)}_t&=H^{(x^1,x^2)(y^1,y^2)}_t-\int_0^t\, H^{(x^1,x^2)}_u\lambda^{(x^1,x^2)(y^1,y^2)}_u\, du\\
&=H^{(x^1,x^2)(y^1,y^2)}_t-\int_0^t\, H^{1; x^1}_uH^{2; x^2}_u\lambda^{(x^1,x^2)(y^1,y^2)}_u\, du,\quad t \in [0,T].  \end{align*}
 In view of \cite[Theorem 2.8]{BieJakNie2014a} %\ref{thm:int-comp}
 process $K^{(x^1,x^2)(y^1,y^2)}$ is an $\FF \vee \FF^X\,$--$\,$local martingale. Since, in view of Assumption (A), $X$ is also an $(\FF, \FF^X)$-DSMC, then, Theorem 4.10 in \cite{BieJakNie2014a} implies  that $K^{(x^1,x^2)(y^1,y^2)}$ is also $\widehat{\FF}^{X}\,$--$\,$local martingale, where $\widehat{\FF}^{X}:= (\F_T \vee \F^X_t)_{t \in [0,T] }$. Consequently,
process $K^{x^1y^1}$ given as
\begin{align}\label{eq:M-X1}
	K^{x^1y^1}_t =\sum_{x^2, y^2 \in S_2 }K^{(x^1,x^2)(y^1,y^2)}_t= H^{1;x^1y^1}_t - \int_0^t \sum_{x^2, y^2 \in S_2 } H^{1; x^1}_uH^{2; x^2}_u \lambda^{(x^1,x^2)(y^1,y^2)}_u du
\end{align}
is an $\FF \vee \FF^X\,$--$\,$local martingale as well as an $\widehat{\FF}^{X}\,$--$\,$local martingale.

\noindent \underline{Step 2:}

Now, assume that  \eqref{weak-M1} holds. Then, using \eqref{eq:M-X1} we obtain that
\begin{align*}
	K^{x^1y^1}_t &= H^{1;x^1y^1}_t - \int_0^t  H^{1;x^1}_u \sum_{x^2 \in S_2} \left[ H^{2;x^2}_u \left( \sum_{y^2 \in S_2}  \lambda^{(x^1,x^2)(y^1,y^2)}_u \right) \right]du \\
	&= H^{1;x^1y^1}_t - \int_0^t  H^{1;x^1}_u \left[ \sum_{y^2 \in S_2}  \lambda^{(x^1,X^2_u)(y^1,y^2)}_u \right]du \\
 &= H^{1;x^1y^1}_t - \int_0^t  H^{1;x^1}_u  \lambda^{1;x^1y^1}_u du.
\end{align*}
Thus, since  $K^{x^1y^1}$ is a $\widehat{\FF}^{X}\,$--$\,$local martingale, then, by \cite[Theorem 4.10]{BieJakNie2014a}%\ref{thm:characterization-DSMC}
, the process $X^1$ is $(\FF, \FF^X)$-DSMC with intensity process $\Lambda^1$.
$X$ is an $(\FF, \FF^X)$-DSMC, so the filtration $\FF$ is immersed in $\FF \vee \FF^X$ (see \cite[Corollary 4.7]{BieJakNie2014a}).  %\ref{cor:hihihaha}
Consequently, applying \cite[Proposition 4.13]{BieJakNie2014a} %\ref{prop:DSMCisCMC}
we conclude that $X^1$ is $(\FF, \FF^X)$-CMC.

\noindent \underline{Step 3:}

Conversely, assume that $X^1$ is an $(\FF,\FF^X)$-CMC with $\FF$-intensity $\Lambda^1$. So, process  $\widehat{K}^{x^1y^1}$ given as
\[
\widehat{K}^{x^1y^1}_t = H^{1;x^1y^1}_t - \int_0^t  H^{1;x^1}_u  \lambda^{1;x^1y^1}_udu,\quad t \in [0,T],
\]
is an $\FF\vee \FF^X\,$--$\,$local martingale. Recall that process  $K^{x^1y^1}$ defined in \eqref{eq:M-X1} is an $\FF \vee \FF^X$--local martingale.
Consequently, the difference $\widehat{K}^{x^1y^1} - K^{x^1y^1}$, which equals
\begin{align*}
\widehat{K}^{x^1y^1}_t - K^{x^1y^1}_t &=
\int_0^t H^{1;x^1}_u \left( \sum_{x^2, y^2 \in S_2 } H^{2;x^2}_u \lambda^{(x^1,x^2)(y^1,y^2)}_u - \lambda^{1;x^1y^1}_u \right) du \\
&=
\int_0^t  H^{1;x^1}_u \left[ \sum_{y^2 \in S_2}  \lambda^{(x^1 X^2_u)(y^1 y^2)}_u - \lambda^{1;x^1y^1}_u \right]du, \quad \tT,
 \end{align*}
  is a continuous $\FF\vee \FF^X\,$--$\,$local martingale of finite variation, and therefore it is equal to the  null process. This implies \eqref{weak-M1}. The proof of the theorem is complete.
\finproof
\end{proof}
%
%\tr{trzeba dopisac, o ile to prawda, ze z dowodu powyzszego twierdzenia wynika, ze wspolrzedna $X^k$ jest $(\mathbb{F},\mathbb{F}^{X})$-CMC, a zatem jest ona $(\mathbb{F},\mathbb{F}^{X^k})$-CMC }
%
%
The next theorem  gives sufficient and necessary conditions for strong Markovian consistency property of $X$   with respect to  $(\mathbb{F},\cY)$. This theorem will be used to prove Proposition \ref{cor:56}, which will be critically important in the study of strong CMC copulae in Section \ref{SCMCC}.
\begin{theorem}\label{cor:54}
%Let $Y=Y^1, \ldots, Y^N$ be processes such that each $Y^k $ is an $(\mathbb{F},\mathbb{F}^{Y^k})$-CMC with values in $S_k$.
Let $\cY=\set{Y^1, \ldots, Y^N}$ be a family of processes such that each $Y^k $ is an $(\mathbb{F},\mathbb{F}^{Y^k})$-CDMC,  with values in $S_k$, and with $\FF$-intensity $\Psi^k_t=[\psi^{k;x^ky^k}_t]_{x^k,y^k \in S_k}$.
Let process $X$ satisfy Assumption (A). Then, $X$ satisfies the strong Markovian consistency property with respect to  $(\mathbb{F},\cY)$ if and only if for all $k=1,2,\ldots,N$, the following hold:
\begin{itemize}
\item[(i)] Condition %\eqref{weak-Mk}
(SM-k) is satisfied with
\[
\lambda^{k;x^ky^k} = \psi^{k;x^ky^k}, \quad
x^k,y^k \in S_k, x^k\neq y^k.
\]
\item[(ii)] The law of $X_0^k$ given $\mathcal{F}_T$ coincides with the law of $Y^k_0$ given $\mathcal{F}_T$. %    for all $k=1,2,\ldots,N$.
\end{itemize}
\end{theorem}
\proof
First we prove sufficiency.  In view of (i) we conclude from Theorem \ref{thm:weak-Mk} that process $X$ is strongly Markovian consistent with respect to $(X^k, \FF) $, and that $X^k$ admits the $\FF$-intensity $\Psi^k$, for each $k=1,2,\ldots,N$.
This, combined with (ii) implies, in view of Lemma \ref{lem:paskudny}, that $X$ satisfies the strong Markovian consistency property with respect to  $(\mathbb{F},\cY)$.\\
Now we prove necessity.
Since $X$ satisfies the strong Markovian consistency property with respect to  $(\mathbb{F},\cY)$, then, clearly, the law of $X_0^k$ given $\mathcal{F}_T$ coincides with the law of $Y^k_0$ given $\mathcal{F}_T$ for all $k=1,2,\ldots,N$. In addition, in view of Theorem \ref{thm:weak-Mk} and Lemma \ref{lem:paskudny}, we conclude that \eqref{weak-Mk} is satisfied with
$\Lambda^k=\Psi^k$, for all $k=1,2,\ldots,N$.
\finproof
\endproof

\subsection{Algebraic Conditions for Strong Markov Consistency}
The necessary and sufficient condition for strong Markov consistency stated in  Theorem \ref{thm:weak-Mk} may not be easily verified. Here, we provide an algebraic sufficient condition for strong Markov consistency, which typically is easily verified.  Towards this end let us fix $k\in \set{1,2,\ldots,N}$, and let us consider the following condition\footnote{The acronym ASM comes from Algebraic Strong Markov.}

%
%\smallskip
%\centerline{\framebox[1.02\width][c]{
%{\sf Condition (ASM-k):$\ $}
%{\it The $\mathbb{F}$-intensity process $\Lambda$ satisfies, for every $t \in [0,T],$ and for all $\bar{x}^n,x^n \in S_n,\, n\ne
%k,\ x^k,\ y^k \in S_k,\, x^k\ne y^k $,
%\begin{align}\label{Mn}
%\sum_{\substack{y^n\in S_n,\\ n=1,2,\ldots,N, n \neq k}}\lambda^{(x^1,\ldots,x^k,\ldots,x^N)(y^1,\ldots,y^k,\ldots,y^N)}_t
%&=
%\sum_{\substack{y^n\in S_n,\\ n=1,2,\ldots,N, n \neq k}}\lambda^{(\bar x^1,\ldots,\bar x^{k\!-\!1},x^k,\bar x^{k\!+\!1},\ldots,\bar x^N)(y^1,\ldots,y^k,\ldots,y^N)}_t
%. \nonumber
%\end{align}
%}
%}
%}
\vskip 10pt
\noindent
{\sf Condition (ASM-k):$\ $} {\it The $\mathbb{F}$-intensity process $\Lambda$ of $X$ satisfies, for every $t \in [0,T],$ and for all $x^k,\ y^k \in S_k,\, x^k\ne y^k $, and
$\bar{x}^n,x^n \in S_n,\, n\ne k,$
\begin{align}\label{Mn}
\sum_{\substack{y^n\in S_n,\\ n=1,2,\ldots,N, n \neq k}}\lambda^{(x^1,\ldots,x^k,\ldots,x^N)(y^1,\ldots,y^k,\ldots,y^N)}_t
&=
\sum_{\substack{y^n\in S_n,\\ n=1,2,\ldots,N, n \neq k}}\lambda^{(\bar x^1,\ldots,\bar x^{k\!-\!1},x^k,\bar x^{k\!+\!1},\ldots,\bar x^N)(y^1,\ldots,y^k,\ldots,y^N)}_t
. \nonumber
\end{align}
}

\begin{remark}
Contrary to Condition (SM-k),  whether condition (ASM-k) holds or not depends on the choice of version of $\FF$ intensity (see Example \ref{ex:aleuszy}).
\end{remark}

We note that Condition (ASM-k) generalizes the analogous condition introduced in Bielecki, Jakubowski, Vidozzi and Vidozzi \cite{BieJakVidVid2008} for Markov chains, and called there Condition (M). The importance of Condition (ASM-k) stems from the next result.
\begin{proposition}\label{suffstrong}
Let process $X$ satisfy Assumption (A), and let us fix $k\in \set{1,2,\ldots,N}$. Then, Condition (ASM-k) is sufficient  for strong Markovian consistency of $X$ relative to $(X^k, \FF)$ and for $\Lambda^k=[\lambda^{k;x^ky^k}]_{x^k,y^k\in S_k}$ to be an $\FF$-intensity process of  $X^k$, where $\lambda^{k;x^ky^k}$ is given as
\begin{equation}\label{lambdeczka}
\lambda^{k;x^ky^k}=\sum_{\substack{y^n\in S_n,\\ n=1,2,\ldots,N, n \neq k}}\lambda^{(x^1,\ldots,x^k,\ldots,x^N)(y^1,\ldots,y^k,\ldots,y^N)}
\end{equation} for $x^k\ne y^k$, and
\[
\lambda^{k;{x^kx^k}}= -\sum_{y^k \in S_k, y^k \neq x^k}\lambda^{k;{x^ky^k}}.
\]
\end{proposition}
\begin{proof}
Condition (ASM-k) implies that for any  $x^k,y^k \in S_k$, $x^k\ne y^k$,  the following sum
\begin{equation}\label{paskudnylabelek}
\sum_{\substack{y^n\in S_n,\\ n=1,2,\ldots,N, n \neq k}}\lambda^{(x^1,\ldots,x^k,\ldots,x^N)(y^1,\ldots,y^k,\ldots,y^N)}_t,
\end{equation}
does not depend on $x^1, \ldots, x^{k-1}, x^{k+1}, \ldots, x^N$.  Thus, condition \eqref{weak-Mk} holds with $\lambda^{k;x^ky^k}$ given by \eqref{lambdeczka}.
 Consequently, the result follows by application of Theorem \ref{thm:weak-Mk}.
\finproof
\end{proof}
Proposition  \ref{cor:56} below will play the key role in Section \ref{SCMCC}.
\begin{proposition}\label{cor:56}
Let $\cY=\set{Y^1, \ldots, Y^N}$ be a family of processes such that each $Y^k $ is an $(\mathbb{F},\mathbb{F}^{Y^k})$-CDMC with values in $S_k$, and with $\FF$-intensity $\Psi^k_t=[\psi^{k;x^ky^k}_t]_{x^k,y^k \in S_k}$. Let process $X$ satisfy Assumption (A). Assume that
\begin{itemize}
\item[(i)]  There exists a version of $\FF$--intensity $\Lambda$ which satisfies the following condition:

 for each $k=1,2,\ldots,N$, $x^k,y^k\in S_k,\ x^k\ne y^k$,
\begin{equation}\label{eq:gambdeczka}
\psi^{k;x^ky^k}_{t}=\sum_{\substack{y^n\in S_n,\\ n=1,2,\ldots,N, n \neq k}}\lambda^{(x^1,\ldots,x^k,\ldots,x^N)(y^1,\ldots,y^k,\ldots,y^N)}_t.
\end{equation}
\item[(ii)] The law of $X_0^k$ given $\mathcal{F}_T$ coincides with the law of $Y^k_0$ given $\mathcal{F}_T$ for all $k=1,2,\ldots,N$.
\end{itemize}
Then, $X$ satisfies the strong Markovian consistency property with respect to  $(\mathbb{F},\cY)$.
\end{proposition}
\begin{proof}
We observe, that  for  $\FF$--intensity $\Lambda$  satisfying $(i),$
Condition (ASM-k) holds for every  $k=1,2,\ldots,N$.
Thus, by Proposition \ref{suffstrong} it follows that  \eqref{weak-Mk} holds with $\lambda^{k;x^ky^k} = \psi^{k;x^ky^k}$, $\forall x^k,y^k\in S_k,\ x^k\ne y^k$. From (ii)
%assumption that conditional laws of $X^k_0$ coincide with conditional laws of $Y^k$ for every $k=1,2,\ldots,N$
and Theorem \ref{cor:54}, we conclude that
$X$ is strongly Markovian consistent with respect to  $(\mathbb{F},\cY)$.\finproof
\end{proof}

\subsubsection{Condition (ASM-k) is not necessary for strong Markovian consistency} \label{niezakoniecznie}

Example \ref{ex:alejaja} below shows that, in general,  Condition (ASM-k) is not necessary for strong Markovian consistency of $X$  relative to $(X^k, \FF).$ Thus, Condition (SM-k) is (essentially) weaker than Condition (ASM-k). In fact, Condition (ASM-k) is so powerful that it implies strong Markovian consistency of $X$  relative to $(X^k, \FF)$ regardless of the  initial distribution of  process $X$. However, whether or not Condition (SM-k) holds depends also on the initial distribution of $X$.\footnote{
 This observation suggests that the relation between Condition (ASM-k) and Condition (SM-k) is analogous to the
relationship between  strong lumpability property and weak lumpability property (cf. Ball and Yeo \cite{BalYeo1993}, Burke and Rosenblatt \cite{BurRos1958}).}

\begin{example}\label{ex:alejaja}
Consider a bivariate process  $X=(X^1,X^2)$ taking values in a finite state space $S=\set{(0,0),(0,1),(1,0),(1,1)}$, and such  that it is an $(\FF,\FF^X)$--CDMC. Assume that $X$ admits the $\FF$-intensity $\Lambda$ of the form\,\footnote{It was shown in \cite[Proposition 4.13]{BieJakNie2014a} that one can always construct an $(\FF, \FF^X)$--CDMC with a given $\FF$-intensity $\Lambda$.}
\begin{equation}\label{eq:lambda-przyk}
	\Lambda_t =[\lambda^{xy}_t]_{x,y\in S}=
\bordermatrix{&(0,0)&(0,1)&(1,0)&(1,1)\cr
  (0,0) &  -a_t  & 0 & 0 & a_t \cr
  (0,1) & 0 & 0 & 0 & 0 \cr
  (1,0) & 0 & 0 & 0 & 0 \cr
  (1,1) & b_t & 0 & 0 & -b_t }.
\end{equation}
Let us
suppose that $\F_T$--conditional distribution of $X_0$ is given as
\begin{equation}\label{eq:init-dist-example}
\begin{gathered}
\P(X_0=(0,1)| \mathcal{F}_T)
=
\P(X_0=(1,0)| \mathcal{F}_T)=0,
\\
\P(X_0=(0,0)| \mathcal{F}_T)
=m_0, \quad
\P(X_0=(1,1)| \mathcal{F}_T)=m_1,
\end{gathered}
\end{equation}
where $m_0$, $m_1$ are $\F_0$ measurable random variables.

Now let us investigate  Condition (SM-1) relative to this $X$.
One can verify that c--transition field of $X$ (see \cite[Definition 4.4]{BieJakNie2014a} for the concept of c--transition field) has the following structure
\[
\bordermatrix{&(0,0)&(0,1)&(1,0)&(1,1)\cr
  (0,0) &  P^1_{00}(s,t)  & 0 & 0 & P^1_{01}(s,t) \cr
  (0,1) & 0 & 1 & 0 & 0 \cr
  (1,0) & 0 & 0 & 1 & 0 \cr
  (1,1) & P^1_{10}(s,t) & 0 & 0 & P^1_{11}(s,t) }.
\]
Thus, in view of \cite[Proposition 4.6]{BieJakNie2014a}, we conclude that for any $\tT$
\be\label{eq:P01}
\P( X_t= (0,1) | \F_T  )=  \P( X_t= (1,0) | \F_T  )=0,
\ee
and that
\[
\P( X_t= (0,0) | \F_T  ) = m_0 P^1_{00}(0,t) + m_1 P^1_{10}(0,t).
\]
Consequently, as we will show now Condition (SM-1) (i.e. \eqref{weak-Mk} for $k =1$)
 is satisfied here.
In fact, taking $x^1 = 0, y^1 = 1$ and invoking \eqref{eq:init-dist-example},  we obtain that
\begin{align*}
%x^1 = 0
%x^2 = 0
%y^1 = 1
& \I_\set{ X^1_t = 0} (\lambda^{(0,X^2_t)(1,0)}_t + \lambda^{(0,X^2_t)(1,1)}_t)
\\
&
=\I_\set{ X^1_t = 0, X^2_t =0}
(\lambda^{(0,0)(1,0)}_t + \lambda^{(0,0)(1,1)}_t )+
%x^1 = 0
%x^2 = 1
%y^1 = 1
\I_\set{ X^1_t = 0, X^2_t = 1}
(\lambda^{(0,1)(1,0)}_t + \lambda^{(0,1)(1,1)}_t )
\\
&=
\I_\set{ X^1_t = 0, X^2_t =0}
a_t
+
\I_\set{ X^1_t = 0, X^2_t = 1} 0
=
(\I_\set{ X^1_t = 0, X^2_t =0}
+
\I_\set{ X^1_t = 0, X^2_t = 1} ) a_t  =\I_\set{ X^1_t = 0} a_t,
\end{align*}
where the third equality follows from the fact that
%under \eqref{eq:init-dist-example} it holds
\begin{equation}\label{zerunko}
\I_\set{ X^1_t = 0, X^2_t = 1} = 0 , \quad dt \otimes d \P,
\end{equation}
which is a consequence of \eqref{eq:P01}.
 Analogously, for $x^1 = 1, y^1 = 0$ it holds that
\begin{align*}
%x^1 = 1
%x^2 = 0
%y^1 = 0
& \I_\set{ X^1_t = 1} (\lambda^{(1,X^2_t)(0,0)}_t + \lambda^{(1,X^2_t)(0,1)}_t)
\\
&
=\I_\set{ X^1_t = 1, X^2_t =0}
(\lambda^{(1,0)(0,0)}_t + \lambda^{(1,0)(0,1)}_t )
+
\I_\set{ X^1_t = 1, X^2_t = 1}
(\lambda^{(1,1)(0,0)}_t + \lambda^{(1,1)(0,1)}_t )
\\
&=
\I_\set{ X^1_t = 1, X^2_t =0}
0
+
\I_\set{ X^1_t = 1, X^2_t = 1} b_t
=
(\I_\set{ X^1_t = 1, X^2_t =0}
+
\I_\set{ X^1_t = 1, X^2_t = 1} ) b_t
=\I_\set{ X^1_t = 1} b_t,
\end{align*}
where we used the fact that
\begin{equation}\label{zerunko2}
\I_\set{ X^1_t = 1, X^2_t = 0} = 0 , \quad dt \otimes d \P.
\end{equation}
Thus, Condition  (SM-1) %\eqref{weak-M1}
holds here for
\[
\lambda^{1;01}_t = a_t, \quad
\lambda^{1;11}_t = b_t.
\]
Similarly,  one can show that Condition (SM-2) is fulfilled for
\[
\lambda^{2;01}_t = a_t, \quad
\lambda^{2;11}_t = b_t.
\]
Thus, $X$ is strongly Markovian consistent with respect to $\FF$.
However, Condition (ASM-1) is not satisfied here (regardless of the initial distribution of $X$). This is because for every $\tT$ we have that
\[
\lambda^{(0,0)(1,0)}_t+\lambda^{(0,0)(1,1)}_t=a_t
\neq
0 =
\lambda^{(0,1)(1,0)}_t+\lambda^{(0,1)(1,1)}_t.
\]
\finproof
\end{example}

\begin{remark}
It is worth noting that strong Markovian consistency depends on the initial distribution of $X$.
{ Consequently, we may have two processes: $X$ which is $(\FF,\FF^X)$-CMC  and $Y$  which is $(\FF,\FF^Y)$-CMCs with the same $\FF$-intensity, such that one of them is strongly Markovian consistent and the other one is not.}
%\cancel{ Consequently, we may have two $(\FF,\GG)$-CMCs}
%\cancel{  with the same $\FF$-intensity,
% such that one of them is strongly Markovian consistent and the }
%\cancel{   other one is not.}
To see this, note that \eqref{weak-Mk}  would not be satisfied  for {$(\FF,\FF^Y)$}-CMC process $Y$ taking values in a finite state space $S=\set{(0,0),(0,1),(1,0),(1,1)}$, endowed with the same  $\FF$-intensity as in Example \ref{ex:alejaja}, and with the conditional initial distribution such that either
\[
\P( \P(Y_0=(0,1)| \mathcal{F}_T) >0 ) >0,
\]
or
\[\P( \P(Y_0=(1,0)| \mathcal{F}_T) > 0 ) >0.\]
Indeed, for process $Y$ equality \eqref{zerunko} is not satisfied, and thus Condition (SM-1) does not hold.
Consequently, process $Y$ is not strongly Markovian consistent with respect to $\FF$.
%
%The example illustrates, in particular, that whether Condition (SM-k) is satisfied or not for a given $(\FF, \FF^X)$-DCMC X depends on the initial distribution of $X$.
\end{remark}

In the next example we will show that an {$(\FF,\FF^X)$}-CMC $X$ may have intensity for which Condition (ASM-k) does not hold, and it may admit another version of intensity, in the sense of \cite[Definition 2.5.]{BieJakNie2014a}, for which Condition (ASM-k) is fulfilled.
\begin{example}\label{ex:aleuszy}
Let us take $X$ as in Example \ref{ex:alejaja}. In that example we proved that Conditions (ASM-1) and (ASM-2) are not satisfied by the $\FF$-intensity $\Lambda$ given by \eqref{eq:lambda-przyk}. However there exists
another  version of $\FF$-intensity of $X$, say $\Gamma$, for which Conditions (ASM-1) and (ASM-2)  are satisfied.  Indeed, let us consider process $\Gamma$ defined by
\[
	\Gamma_t
=[\gamma^{xy}_t]_{x,y\in S}=
\bordermatrix{&(0,0)&(0,1)&(1,0)&(1,1)\cr
  (0,0) &  -a_t  & 0 & 0 & a_t \cr
  (0,1) & b_t & -a_t - b_t & 0 & a_t \cr
  (1,0) & b_t & 0 & -a_t - b_t & a_t \cr
  (1,1) & b_t & 0 & 0 & -b_t }.
\]
By \cite[Proposition 2.6 (ii)]{BieJakNie2014a} the process  $\Gamma$ is an $\FF$-intensity of  $X$, because in view of \eqref{zerunko} and \eqref{zerunko2} it holds that
\[
\int_0^t (\Gamma_u - \Lambda_u)^\top H_u du = 0, \quad \tT.
\]
Finally, we see that Conditions (ASM-1) and (ASM-2) are satisfied for $\Gamma$, since
\[
\gamma^{(0,0)(1,0)}_t+\gamma^{(0,0)(1,1)}_t=a_t
=\gamma^{(0,1)(1,0)}_t+\gamma^{(0,1)(1,1)}_t,
\]
\[
\gamma^{(1,1)(0,0)}_t+\gamma^{(1,1)(0,1)}_t=b_t
=\gamma^{(1,0)(0,0)}_t+\gamma^{(1,0)(0,1)}_t,
\]
\[
\gamma^{(0,0)(0,1)}_t+\gamma^{(0,0)(1,1)}_t=a_t
=\gamma^{(1,0)(0,1)}_t+\gamma^{(1,0)(1,1)}_t,
\]
\[
\gamma^{(1,1)(0,0)}_t+\gamma^{(1,1)(1,0)}_t=b_t
=\gamma^{(0,1)(0,0)}_t+\gamma^{(0,1)(1,0)}_t.
\]
\end{example}

\section{Weak Markovian Consistency of Conditional Markov Chains}

As in Section \ref{sec:SMC} let us consider $X=(X^1,\ldots,X^N)$ -- a multivariate ${(\FF,\FF^X)}$-CMC with values in $S:=\text{\sf X}_{k=1}^N\, S_k$ (recall that $S_k$ is a finite set, $k=1,\ldots,N$), and admitting an $\FF$-intensity $\Lambda$.

We will study here the concept of weak Markovian consistency. In many respects, this concept is more important in practical applications than the concept of strong Markovian consistency. As it will be seen below, the definitions and results regarding the weak Markovian consistency to some extent parallel those regarding the strong Markovian consistency. But, as always, ``the devil is in the details'', so the reader is kindly asked to be patient with presentation that follows.

\begin{definition}
(i) Let us fix $k\in\set{1,\ldots,N}$. We say that the process $X$ satisfies the  weak Markovian consistency property relative to $(X^k, \, \mathbb{F})$ if  for every $x^k_1,  \ldots, x^k_m\in S_k $
and for all $0 \leq t \leq t_1 \leq \ldots \leq t_m \leq T,$ it holds that
\begin{equation}\label{eq:markov-cons-weak}
\mathbb{P}\left( X^k_{t_m } =x^k_{m},  \ldots, X^k_{t_1 } =x^k_{1}| \mathcal{F}_t \vee  \mathcal{F}^{X^k}_t \right) =   \mathbb{P} \left( X^k_{t_m } =x^k_{m},  \ldots, X^k_{t_1 } =x^k_{1}  | \mathcal{F}_t \vee \sigma(X^k_t)    \right),
\end{equation}
or, equivalently, if  $X^k$ is a $(\mathbb{F}, \mathbb{\FF}^{X^k})$-CMC.

\noindent (ii) If  $X$ satisfies the
 weak Markovian consistency property with respect to $(X^k, \mathbb{F})$ for all $k\in\set{1,\ldots,N}$, then we say that  $X$ satisfies the
 weak Markovian consistency property with respect to $\mathbb{F}$.
\end{definition}

 \begin{definition}
Let $\cY=\set{ Y^1, \ldots, Y^N}$ be a family of processes such that each $Y^k $ is an $(\mathbb{F},\mathbb{F}^{Y^k})$-CMC with values in $S_k$.\\
(i) Let us fix $k\in \set{1,2,\ldots,N}$ and let the process $X$ satisfy the weak Markovian consistency property with respect to  $(X^k, \mathbb{F})$.  If the conditional law of $X^k$ given $\mathcal{F}_T$ coincides with the conditional law of $Y^k$ given $\mathcal{F}_T$, then we say that   $X$ satisfies the weak Markovian consistency property with respect to  $(X^k,\mathbb{F}, Y^k)$.\\
(ii)  If  $X$ satisfies the weak Markovian consistency property with respect to  $(X^k,\mathbb{F}, Y^k)$  for every $k\in \set{1,2,\ldots,N}$, then we say that $X$ satisfies the weak Markovian consistency property with respect to  $(\mathbb{F},\cY)$.
\end{definition}

\subsection{Sufficient and necessary conditions for weak Markovian consistency~I}
 We postulate in this subsection that the process $X$ satisfies Assumption (A) (see Section \ref{sekcjazrameczka}).

We aim here at providing a condition characterizing weak Markovian consistency of the process $X$. We start from introducing\footnote{The acronym WM comes from  Weak Markov.}
\vskip 10pt
\noindent
{\sf Condition (WM-k):$\ $}
There exist $\FF$-{adapted} processes $\lambda^{k;x^ky^k}$, $x^k,y^k \in S_k$, $x^k\neq y^k$, such that
\begin{equation}\lab{vn}
\begin{aligned}
\I_{\{X^k_{t}=x^k\!\}}\!\!\!\!\sum_{\substack{x^n,y^n\in {S}_n \\n=1,2,\ldots,N, n \neq k }} & \!\!\!\!\lambda^{(x^1, \ldots, x^N)(y^1, \ldots, y^N)}_t
\EP \left(\I_{\{X^1_t=x^1, \ldots, X^{k-1}_t=x^{k-1},X^{k+1}_t=x^{k+1},\ldots, X^N_t=x^N\}} \Big\vert  {\mathcal F}_t \vee {\mathcal F}^{X^k}_t\!\right)
\\
& \!=\! \I_{\{X^k_{t}=x^k\}}\lambda^{k;x^ky^k}_t ,
 \quad {dt \otimes
d\mathbb{P} \text{-a.e.} }  \  \forall x^k\!,y^k\!\in\! {S}_k,\ x^k \neq y^k.
\end{aligned}
\end{equation}

Similarly as in the case of Condition (SM-k) we have the following proposition
\begin{proposition}
Let $X$ satisfy Assumption (A) and $\Lambda$, $\wh{\Lambda}$ be $\FF$-intensities of $X$. Then (WM-k)  holds for $\Lambda$ if and only if it holds for $\wh{\Lambda}$.
\end{proposition}
\noindent
The proposition is a direct consequence of \cite[Proposition 2.6]{BieJakNie2014a}.

\smallskip
 The next theorem characterizes weak Markovian consistency in the present set-up.

\bt\lab{wmc}
The process $X$  with an $\FF$-intensity $\Lambda$ is weakly Markovian consistent relative to $(X^k,\FF)$ if and only if Condition (WM-k) is satisfied.
 Moreover,
 $X^k$ admits an $\mathbb{F}$-intensity process
\be\label{eq:F-int-lak}
\Lambda^k :=[\lambda^{k;x^ky^k}]_{x^k,y^k\in {S}_k},
\ee
   with
$\lambda^{k;x^kx^k}$ given by
\[
    \lambda^{k;x^kx^k}_t = - \sum_{y^k \in {S}_k, y^k \neq x^k } \lambda^{k;x^ky^k}_t,
    \qquad
    \forall x^k \in {S}_k
    .
\]
\et
%
%\bt\lab{wmc} The component $X^1$ of $X$ is a $(\FF, \FF^{X^1})$-CMC
%if and only if
%\begin{align}\lab{vn}
%\I_{\{X^1_{t}=x^1\!\}}\!\!\!\!\sum_{x^2,y^2\in {S}_2} \!\!\!\!\lambda^{(x^1x^2)(y^1y^2)}_t
%\EP \left(\I_{\{X^2_t=x^2\}}|{\mathcal F}_t \vee {\mathcal F}^{X^1}_t\!\right)
%%\mathbb{P}\left(X^2_t=x^2|{\mathcal F}^{X^1}_t\!\right)
%\!=\! \I_{\{X^1_{t}=x^1\}}\lambda^{1;x^1y^1}_t \quad {dt \otimes
%d\mathbb{P} \text{-a.s.} }  \  \forall x^1\!,y^1\!\in\! {\mathcal
%X}^1\!,\ x^1 \neq y^1
%\end{align}
%for some locally integrable $\FF$-adapted processes $\lambda^{1;x^1y^1}$ The
%intensity process of $X^1$ is $\Lambda^1 (t)
%=[\lambda^{1;x^1y^1}_t]_{x^1,y^1\in S_1}$   with
%$\lambda^{1;x^1x^1}$ given by
%\[
%    \lambda^{1;x^1x^1}_t = - \sum_{y^1 \in {S}^1, y^1 \neq x^1 } \lambda^{1;x^1y^1}_t
%    \qquad
%    \forall x^1 \in \mathcal{S}_1
%    .
%\]
%\et
\begin{proof}
%
%\Mn{dostatecznosci takiego warunku nie potrafie wywniowskowac/udowodnic  dla konstrukcji z poprzedniego rozdzialu. }
%\Mn{Tutaj dowod przy pomocy twierdzenia charakteryzujacego ktore wymaga zalozen 1) immersji $\FF$ w $\FF^{X^1}$ oraz  2) ortogonalnosci $\FF$ martyngalow z martyngalem $M^1$ odpowiadajacym wspolrzednej $X^1$ duzego procesu $X$. TYCH ZALOZEN NIE WPISALEM DO SFORMULOWANIA TWIERDZENIA.}
For simplicity of notation we give proof for $k=1$ and $N=2$. In this case, \eqref{vn} takes the following form (recall our notation: $H^{k;x^k}_t=\I_{\{X^k_{t}=x^k\!\}}$)
\begin{align}\lab{vn-1}
H^{1;x^1}_t\!\!\!\!\sum_{x^2,y^2\in {S}_2} \!\!\!\!\lambda^{(x^1x^2)(y^1y^2)}_t
\EP \left(H^{2;x^2}_t|{\mathcal F}_t \vee {\mathcal F}^{X^1}_t\!\right)
\!=\! H^{1;x^1}_t\lambda^{1;x^1y^1}_t, \quad {dt \otimes
d\mathbb{P} \text{-a.e.} }  \  \forall x^1\!,y^1\!\in\! {\mathcal
X}^1\!,\ x^1 \neq y^1.
\end{align}

\noindent \underline{Step 1:}

In Step 1 of the proof of Theorem \ref{thm:weak-Mk} we have shown that the process $K^{x^1y^1}$ given in \eqref{eq:M-X1} is an $\FF \vee \FF^X\,$--$\,$local martingale.
 Now let us denote by $\widetilde{K}^{x^1y^1}$  the optional projection of $K^{x^1y^1}$ on the filtration $\FF \vee \FF^{X^1}$.\footnote{We note that for existence of optional projections we do not need  right continuity of the filtration (see Ethier, Kurtz \cite[Theorem 2.4.2]{ethkur1986}).}
Observe that the sequence
\[
\tau_n := \inf { \bigg\{ t \geq 0 : H^{1;x^1y^1}_t \geq n \ \textrm{or} \ \int_0^t  \bigg( \sum_{x^2, y^2 \in S_2 }  \lambda^{(x^1,x^2)(y^1,y^2)}_u \bigg) du
\geq  n  \bigg\} },\ n=1,2,\ldots,
\]
is sequence of $\FF \vee \FF^{X}$--stopping times, as well as $\FF \vee \FF^{X^1}$--stopping times, and that is a reducing sequence
for $K^{x^1y^1}$. So, by \cite[Theorem 3.7]{FolProt2011}, the process  $\widetilde{K}^{x^1y^1}$ is an $\FF \vee \FF^{X^1}\,$--$\,$ local martingale.
Following the reasoning in Theorem 5.25 in \cite{HeWanYan1992}, we obtain that
\[
\E\bigg( \int_0^t  H^{1;x^1}_u \!\!\!\! \sum_{x^2, y^2 \in S_2 } \!\!\!\!  H^{2;x^2}_u \lambda^{(x^1,x^2)(y^1,y^2)}_u du \Big | \F_{t} \vee \F^{X^1}_t \bigg)
=
\int_0^t  H^{1;x^1}_u \E\bigg(\sum_{x^2, y^2 \in S_2 } \!\!\!\!  H^{2;x^2}_u \lambda^{(x^1,x^2)(y^1,y^2)}_u \Big | \F_{u} \vee \F^{X^1}_u \bigg)du
\]
and hence the process $\widetilde{K}^{x^1y^1}$ given as
%    Now let us denote by $\widetilde{K}^{x^1y^1}$  the optional projection of $K^{x^1y^1}$ on the filtration $\FF \vee \FF^{X^1}$. Since there exist a reducing sequence
%for $K^{x^1y^1}$
%of $\FF \vee \FF^{X}$ stopping times which are also $\FF \vee \FF^{X^1}$ stopping times we have, by \cite[Theorem 3.7]{FolProt2011}, that  $\widetilde{K}^{x^1y^1}$ is an $\FF \vee \FF^{X^1}\,$--$\,$local martingale. As such reducing sequence we can take  \[
%\tau_n := \inf { \bigg\{ t \geq 0 : H^{1;x^1y^1}_t \geq n \ \textrm{or} \ \int_0^t  \bigg( \sum_{x^2, y^2 \in S_2 }  \lambda^{(x^1,x^2)(y^1,y^2)}_u \bigg) du > n  \bigg\} }.
%\]
%Thus, in view of Theorem 5.25 in He,  Wang and Yan \cite{HeWanYan1992}, process $\widetilde{K}^{x^1y^1}$ given as
\begin{equation}\label{falka}
\widetilde{K}^{x^1y^1}_t = H^{1;x^1y^1}_t - \int_0^t  H^{1;x^1}_u \E\bigg(\sum_{x^2, y^2 \in S_2 } H^{2;x^2}_u \lambda^{(x^1,x^2)(y^1,y^2)}_u \Big | \F_{u} \vee \F^{X^1}_u \bigg)du,\quad t \in [0,T].
\end{equation}
%Since projections of local martingales are not necessarily local martingales.
is an $\FF \vee \FF^{X^1}\,$--$\,$local martingale.
%\footnote{$\widetilde{K}^{x^1y^1}$ is the optional projection of $K^{x^1y^1}$ on the filtration $\FF \vee \FF^{X^1}$. }

\noindent \underline{Step 2:}

Now, suppose that \eqref{vn-1} holds, then we have that
\[
\widetilde{K}^{x^1y^1}_t = H^{1;x^1y^1}_t - \int_0^t  H^{1;x^1}_u \lambda^{1;x^1y^1}_u du,\quad t \in [0,T].
\]
Thus according to \cite[Remark 2.9]{BieJakNie2014a} we can apply \cite[Theorem 2.8]{BieJakNie2014a} %\ref{thm:int-comp}
 to process $X^1$ in order to conclude that  $\Lambda^1$ is an $\FF$-intensity of $X^1$, so that $\r^d$-valued process ${\widetilde M}^1=(\wt{M}^{1;x^1} ;{x^1\in S_1})^\top,$ given as
\[
    {\widetilde M}^1_t= H^1_t - \int_0^t (\Lambda^1_u)^{\!\top} H^1_u du, \quad t \in [0,T],
\]
is an $\FF \vee \FF^{X^1}\,$--$\,$local martingale.

Next, using \cite[Theorem 2.11]{BieJakNie2014a}, %\ref{thm:CMC-char}
 we will show   that $X^1$ is an $(\FF , \FF^{X^1})$-CMC. Towards this end, we first observe that Assumption (A) implies, by \cite[Corollary 4.7]{BieJakNie2014a}, that $\FF$ is immersed in $\FF \vee \FF^{X}$, and thus $\FF$ is immersed $\FF \vee \FF^{X^1}$.   Moreover, as we will show now, all real valued $\FF$-local martingales are orthogonal to  processes $M^x, x\in S$, that are components of process $M$  defined in \eqref{eq:Mart-M-nowy}. Indeed, let us take an arbitrary real valued $\FF$-local martingale $N$. Then, by definition of $M$ and the fact that $M$ is a pure-jump local martingale, we have, for any $(x^1,x^2)\in S$,
\begin{align}\label{eq:NM-ort}
[N, M^{(x^1,x^2)}]_t
=\sum_{ 0< u \leq t } \Delta N_u \Delta M^{(x^1,x^2)}_u
=\sum_{ 0< u \leq t } \Delta N_u \Delta H^{(x^1,x^2)}_u,\quad t \in [0,T].
\end{align}
Now, since the jump times of $N$ are $\FF$-stopping times, then by Proposition 6.1 in \cite{JakNie2010} we conclude  that $N$ and $X$ do not have common jump times, or, equivalently, $N$ and $M$, do not have common jump times. Therefore, $[N, M^{x^1,x^2}] = 0$, so that $N$ is orthogonal to  processes $M^{x^1,x^2}$.

From the above we will deduce that all real valued $\FF$-local martingales  are orthogonal to processes $\wt{M}^{1;x^1}, x^1\in S_1$, that  are components of process $\wt{M}^{1}$ defined above.
In fact taking $N$  as above we see that orthogonality of $N$ and $\wt{M}^{1;x^1}$ follows from the following reasoning
\begin{align*}
[N, \wt{M}^{1;x^1}]_t &= \sum_{ 0< u \leq t } \Delta N_u \Delta \wt{M}^{1;x^1}_u = \sum_{ 0< u \leq t } \Delta N_u \Delta H^{1;x^1}_u
=
%\mn{napewno!!!}
\sum_{ 0< u \leq t } \sum_{x^2 \in S_2} \Delta N_u \Delta H^	{(x^1,x^2)}_u
\\
&=
\sum_{x^2 \in S_2} \sum_{ 0< u \leq t } \Delta N_u \Delta H^{(x^1,x^2)}_u
=
\sum_{x^2 \in S_2} [N, M^{(x^1,x^2)}]_t =0,\quad t \in [0,T],
\end{align*}
where the penultimate equality follows from \eqref{eq:NM-ort}.

Consequently, we see that assumptions of \cite[Theorem 2.11]{BieJakNie2014a} %\ref{thm:CMC-char}
 are fulfilled (taking there $X=X^1$ and $\GG=\FF^{X^1}$), and thus we may conclude that $X^1$ is $(\FF , \FF^{X^1})$-CMC with $\FF$-intensity $\Lambda^1 (t) =[\lambda^{1;x^1y^1}_t]_{x^1,y^1\in S_1}$.

\noindent \underline{Step 3:}

Conversely, assume that $X^1$ is an $(\FF,\FF^{X^1})$-CMC with $\FF$-intensity $\Lambda^1 (t) =[\lambda^{1;x^1y^1}_t]_{x^1,y^1\in S_1}$. Fix $x^1, y^1 \in S_1$, $x^1 \neq y^1$. Since $\Lambda^1$ is an $\FF$-intensity of $X^1$, then the process $\widehat{K}^{x^1y^1}$ given as
\[
\widehat{K}^{x^1y^1}_t = H^{1;x^1y^1}_t - \int_0^t  H^{1;x^1}_u  \lambda^{1;x^1y^1}_udu, \quad t \in [0,T],
\]
is an $\FF\vee \FF^{X^{1}}\,$--$\,$local martingale. Recalling  that process $\widetilde{K}^{x^1y^1}$ given in \eqref{falka} is an $\FF \vee \FF^{X^1}\,$--$\,$local martingale, we see that the difference $\widetilde{K}^{x^1y^1} - \wh{K}^{x^1y^1}$, which is given by
\[
\widetilde{K}^{x^1y^1}_t - \wh{K}^{x^1y^1}_t =
\int_0^t\, H^{1;x^1}_u
\left(
\sum_{x^2, y^2 \in S_2 } \lambda^{(x^1,x^2)(y^1,y^2)}_u \E\bigg(H^{2;x^2}_u
| \F_{u} \vee \F^{X^1}_u \bigg)
 - \lambda^{1;x^1y^1}_u \right) du,
 \]
  is a continuous $\FF \vee \FF^{X^1}\,$--$\,$local martingale of finite variation, and therefore
is equal to $0$. This implies \eqref{vn-1}.

The proof of the theorem is complete.
\finproof
\end{proof}
%{
%\begin{remark}\label{rem:wmc}
%If the process $X$ is weakly Markovian consistent relative to $(X^k, \FF)$ and if it admits an $\FF$-intensity $\Gamma$, then $\Gamma$ satisfies Condition (WM-k) and processes $\gamma^{k;x^ky^k}, x^k, y^k \in S_k, x_k \neq y_k$ appearing in Condition (WM-k) define, in the same way as in Theorem \ref{wmc}, an $\FF$-intensity of $X^k$.
%\end{remark}
%}

%Won do buczka
%$\skull$
%The next theorem  gives sufficient and necessary conditions for weak Markovian consistency property of $X$   with respect to  $(\mathbb{F},\cY)$. This theorem will be used to prove Proposition \ref{cacy}, which will be critically important in the study of weak CMC copulae in Section \ref{sec:wcmc}.
%$\skull $

The next theorem  gives sufficient and necessary conditions for weak Markovian consistency property of $X$   with respect to  $(\mathbb{F},\cY)$.
%This theorem will be used in the study of weak CMC copulae in Section \ref{sec:wcmc}.
We omit the proof of this theorem, as its proof can be derived from the proof of Theorem \ref{cor:54}   by using Theorem \ref{wmc} instead of Theorem \ref{thm:weak-Mk}.

\begin{theorem}\label{wcor:54}
%Let $Y=Y^1, \ldots, Y^N$ be processes such that each $Y^k $ is an $(\mathbb{F},\mathbb{F}^{Y^k})$-CMC with values in $S_k$.
Let $\cY=\set{Y^1, \ldots, Y^N}$ be a family of processes such that each $Y^k $ is an $(\mathbb{F},\mathbb{F}^{Y^k})$-CDMC,  with values in $S_k$, and with $\FF$-intensity $\Psi^k_t=[\psi^{k;x^ky^k}_t]_{x^k,y^k \in S_k}$.
Let process $X$ satisfy Assumption (A) and let $\Lambda$ be a version of its $\FF$-intensity. Then, $X$ satisfies the weak Markovian consistency property with respect to  $(\mathbb{F},\cY)$ if and only if for all $k=1,2,\ldots,N$, the following hold:
\begin{itemize}
\item[(i)] Condition %\eqref{weak-Mk}
(WM-k) is satisfied with
$\Psi^k$ in place of $\Lambda^k$.
\item[(ii)] The law of $X_0^k$ given $\mathcal{F}_T$ coincides with the law of $Y^k_0$ given $\mathcal{F}_T$. %    for all $k=1,2,\ldots,N$.
\end{itemize}
\end{theorem}
%\proof
%First we prove sufficiency.  In view of (i) we conclude from Theorem \ref{wmc} that process $X$ is weakly Markovian consistent with respect to $(X^k, \FF) $, and that $X^k$ admits the $\FF$-intensity $\Psi^k$, for each $k=1,2,\ldots,N$.
%This, combined with (ii) implies, in view of Lemma \ref{lem:paskudny}, that $X$ satisfies the weak Markovian consistency property with respect to  $(\mathbb{F},\cY)$.\\
%Now we prove necessity.
%Since $X$ satisfies the weak Markovian consistency property with respect to  $(\mathbb{F},\cY)$, then, clearly, the law of $X_0^k$ given $\mathcal{F}_T$ coincides with the law of $Y^k_0$ given $\mathcal{F}_T$ for all $k=1,2,\ldots,N$. In addition, in view of Theorem \ref{wmc} and Lemma \ref{lem:paskudny}, we conclude that \eqref{vn} is satisfied with
%$\Lambda^k=\Psi^k$, for all $k=1,2,\ldots,N$.
%\finproof
%\endproof

\subsection{Sufficient and necessary condition for weak Markovian consistency II}

Conditions (WM-k) are mathematically interesting, but they are difficult to verify since they entail computations of projections on the filtration $\FF \vee \FF^{X^k}$. Here we will formulate an ``algebraic like'' necessary condition for weak Markovian consistency, which is easier to verify.
%Won do buczka
%Moreover, this necessary condition will be shown to be a sufficient condition for weak Markovian consistency if an additional assumption (cf. \eqref{con:cacy}) is imposed on $X$.

We start with  imposing the following simplifying assumption on process $X$:

\vskip 10pt
%\centerline{\framebox[1.02\width][c]{
%\framebox{\parbox{\dimexpr\linewidth-2\fboxsep-2\fboxrule}{\itshape%
%{
\begin{center}
\fbox{\begin{varwidth}{\dimexpr\textwidth-2\fboxsep-2\fboxrule\relax}
{\sf Assumption (B):}
For each $k\in \{1,2,\ldots,N\}$ it holds that~
$$
\P \left(X^k_t=x^k \big \vert {\mathcal F}_t \right) > 0, \quad {dt} \otimes
d\mathbb{P} \text{-a.e.}, \quad \forall x^k\in S_k.
$$
\end{varwidth}}
\end{center}
%}
%}
%}
%\vskip 10 pt
%
%for each $k\in \{1,2,\ldots,N\}$
%\be\label{eq:access}
%\P \left(X^k_t=x^k \big \vert {\mathcal F}_t \right) > 0, \quad {\textcolor[rgb]{1.00,0.00,0.50}{dt} \otimes
%d\mathbb{P} \text{-a.e.} }, \quad \forall x^k\in S_k.
%\ee
Clearly, this assumption imposes constraints on the initial distribution of the chain, as well as constraints on the structure of the intensity process of $X$. However, it allows to simplify and to streamline the discussion below. The general case can be dealt with in a similar way, with special attention paid to sets of $\omega$-s for which $\P \left(X^k_t=x^k |{\mathcal F}_t \right)(\omega) = 0.$

Before we proceed we observe that Assumption (B) implies that
\[\P \left(X^k_t=x^k  \right) > 0, \quad {dt  \text{-a.e.} }, \quad \forall x^k\in S.\]
We will also need a simple technical result regarding events $B(t,k,x^k)$ and $C(t,k,x^k)$ defined, for every $\tT$,  $x^k\in S_k$ and $k\in \{1,2,\ldots,N\}$ as
\[
B(t,k,x^k) = \Set{ \omega: X^k_t(\omega) = x^k} , \quad
C(t,k,x^k)=  \Set{ \omega :  \P \left(X^k_t=x^k |{\mathcal F}_t \right)(\omega) > 0 }.
\]
We claim that (we write $B$ and $C$ in place of  $B(t,k,x)$ and $C(t,k,x)$ to shorten the formulae)
\[
\P(B \cap C ) =  \P(B )  > 0.
\]
Indeed,
\[
\P(B \cap C )  = \E \left( \I_{B} \I_{C} \right)
=
\E \left( \E \left( \I_{B} | \F_t \right)  \I_{C} \right)
%=
%\E \left( \P \left( B | \F_t \right)  \I_{C} \right)\]\[
=
\E \left( \P \left( B | \F_t \right)  \I_\set{\P \left( B | \F_t \right) > 0 } \right)
=
\P \left( B \right)  .
\]
We are now ready to state the main result in this section. We recall that $\Lambda$ is an $\FF$-intensity of $X$.
\begin{proposition}
Assume that $X$ satisfies Assumptions (A) and (B). Fix $k \in \set{1, \ldots, N}$. Suppose that $X$ is weakly Markovian consistent relative to $(X^k, \mathbb{F})$. Then,
the $\FF$-intensity $\Lambda^k$ of $X^k$ defined by \eqref{eq:F-int-lak} satisfies
%, for every $x^k\in S_k$ and for \textcolor[rgb]{0.98,0.00,0.00}{$\P$ almost every} $\omega \in B(t,k,x^k)\cap C(t,k,x^k)$
\begin{equation}
\begin{aligned}\lab{vn-2}
 \lambda^{k;x^ky^k}_t(\omega)
&
\!=\! \sum_{\substack{x^n,y^n\in {S}_n \\n=1,2,\ldots,N, n \neq k }}  & \!\!\!\!\lambda^{(x^1, \ldots, x^N)(y^1, \ldots, y^N)}_t(\omega)
\frac{\P \left(X^1_t=x^1, \ldots, X^N_t=x^N|{\mathcal F}_t\right)(\omega)  }{ \P \left(X^k_t=x^k |{\mathcal F}_t \right) (\omega) },
%nonumber
\\
 &&\quad \forall x^k,y^k\!\in\! {S}_k,\ y^k \neq x^k\ \textrm{and} \ \forall \omega \in B(t,k,x^k)\cap C(t,k,x^k),
\end{aligned}
\end{equation}
for almost every $\tT$.
%
%$\lambda^{k;x^kx^k}$ given by
%\[
%    \lambda^{k;x^kx^k}_t = - \sum_{y^k \in {S}_k, y^k \neq x^k } \lambda^{k;x^ky^k}_t,
%    \qquad
%    \forall x^k \in {S}_k
%    .
%\]
%
%
%
%for some  $\FF$-adapted processes $\lambda^{k;x^ky^k}.$
\end{proposition}
\begin{proof}
Since weak Markovian consistency relative to $(X^k, \mathbb{F})$ holds, then $\Lambda^k$  satisfies \eqref{vn}.
Taking conditional expectations with respect to $\F_t \vee \sigma(X^k_t)$ of \eqref{vn} yields
\begin{align*}
	&\I_{\{X^k_{t}=x^k\}}	\lambda^{k;x^ky^k}_t
= \mathbb{E} \left( \I_{\{X^k_{t}=x^k\}}\lambda^{k;x^ky^k}_t | \mathcal{F}_t \vee \sigma(X^k_t) \right)
\\ &=
\bE \bigg(
\I_{\{X^k_{t}=x^k\!\}}\!\!\!\!\sum_{\substack{x^n,y^n\in {S}_n \\n=1,2,\ldots,N, n \neq k }}  \!\!\!\!\lambda^{(x^1, \ldots, x^N)(y^1, \ldots, y^N)}_t \\
& \quad \quad \quad \quad \times \bE \left(\I_{\{X^1_t=x^1, \ldots, X^{k-1}_t=x^{k-1},X^{k+1}_t=x^{k+1},\ldots, X^N_t=x^N\}}|{\mathcal F}_t \vee {\mathcal F}^{X^k}_t\!\bigg)
\bigg| {\mathcal F}_t \vee \sigma(X^k_t)
\right) \\
&\quad =
\I_{\{X^k_{t}=x^k\!\}}\!\!\!\!\sum_{\substack{x^n,y^n\in {S}_n \\n=1,2,\ldots,N, n \neq k }}  \!\!\!\!\lambda^{(x^1, \ldots, x^N)(y^1, \ldots, y^N)}_t
\bE \left(\I_{\{X^1_t=x^1, \ldots, X^{k-1}_t=x^{k-1},X^{k+1}_t=x^{k+1},\ldots, X^N_t=x^N\}}|{\mathcal F}_t \vee \sigma(X^k_t)\!\right).
	\end{align*}
%Since $\I_{\{X^k_{t}=x^k\}}\lambda^{k;x^ky^k}_t $ is measurable with respect to  $\mathcal{F}_t \vee \sigma(X_t)$, then, applying the tower property of conditional expectation applied to right  hand side of \eqref{vn} we have
%\[
%	\mathbb{E} \left( \I_{\{X^k_{t}=x^k\}}\lambda^{k;x^ky^k}_t | \mathcal{F}_t \vee \sigma(X_t) \right)
%=
%\I_{\{X^k_{t}=x^k\}}	\lambda^{k;x^ky^k}_t
%\]
%Thus %, in view of \eqref{vn}, we have  that
%\begin{align*}
%&\I_{\{X^k_{t}=x^k\}}	\lambda^{k;x^ky^k}_t \\
%&=\I_{\{X^k_{t}=x^k\!\}}
%\!\!\!\!\sum_{\substack{x^n,y^n\in {S}_n \\n=1,2,\ldots,N, n \neq k }}  &\!\!\!\!\lambda^{(x^1, \ldots, x^N)(y^1, \ldots, y^N)}_t
%\bE \left(\I_{\{X^1_t=x^1, \ldots, X^{k-1}_t=x^{k-1},X^{k+1}_t=x^{k+1},\ldots, X^N_t=x^N\}}|{\mathcal F}_t \vee \sigma(X^k_t)\!\right)
%\\
%&
%&dt \otimes d \P-a.e, \quad \forall y^k\!\in\! {S}_k,\ y^k \neq x^k.
%\end{align*}
Now, let us take an arbitrary  $\omega \in B(t,k,x^k)\cap C(t,k,x^k)$.
By Assumption (B), using Jakubowski and Niew\k{e}g\l{}owski  \cite[Lemma 3]{JakNie2007}, we have
\begin{align*}
&	\lambda^{k;x^ky^k}_t(\omega) \\
&=
\sum_{\substack{x^n,y^n\in {S}_n \\n=1,2,\ldots,N, n \neq k }}  \!\!\!\!\lambda^{(x^1, \ldots, x^N)(y^1, \ldots, y^N)}_t(\omega)
\bE \left(\I_{\{X^1_t=x^1, \ldots, X^{k-1}_t=x^{k-1},X^{k+1}_t=x^{k+1},\ldots, X^N_t=x^N\}}|{\mathcal F}_t \vee \sigma(X^k_t)\!\right)(\omega)
\\
&=
\sum_{\substack{x^n,y^n\in {S}_n \\n=1,2,\ldots,N, n \neq k }}  \!\!\!\!\lambda^{(x^1, \ldots, x^N)(y^1, \ldots, y^N)}_t(\omega)
\frac{\P \left(X^1_t=x^1, \ldots, X^N_t=x^N|{\mathcal F}_t \right) (\omega)}{ \P \left(X^k_t=x^k |{\mathcal F}_t \right)(\omega) },
\end{align*}
which shows that condition \eqref{vn-2} is necessary for the weak Markovian consistency of $X$ relative to $(X^k, \FF)$.
\finproof
 \end{proof}
The next proposition can be  used in construction of weak CMC copulae.
 \begin{proposition}\label{prop:wmcnec}
Let $\cY=\set{Y^1, \ldots, Y^N}$ be a family of processes such that each $Y^k $ is an $(\mathbb{F},\mathbb{F}^{Y^k})$-CDMC,  with values in $S_k$, and with an $\FF$-intensity $\Psi^k_t=[\psi^{k;x^ky^k}_t]_{x^k,y^k \in S_k}$. Assume that $X$ satisfies Assumptions (A)  and (B),  and let $\Lambda$ be a version of its $\FF$-intensity.
In addition, suppose that $X$ is weakly Markovian consistent relative to $(\mathbb{F}, \cY)$. Then,
\begin{itemize}
\item[(i)] It holds that
%the $\FF$-intensity $\Lambda^k$ of $X^k$ defined by \eqref{eq:F-int-lak} satisfies
%, for every $x^k\in S_k$ and for \textcolor[rgb]{0.98,0.00,0.00}{$\P$ almost every} $\omega \in B(t,k,x^k)\cap C(t,k,x^k)$
\begin{equation}
\begin{aligned}\lab{vn-2a}
 \psi^{k;x^ky^k}_t(\omega)
&
\!=\! \sum_{\substack{x^n,y^n\in {S}_n \\n=1,2,\ldots,N, n \neq k }}  & \!\!\!\!\lambda^{(x^1, \ldots, x^N)(y^1, \ldots, y^N)}_t(\omega)
\frac{\P \left(X^1_t=x^1, \ldots, X^N_t=x^N|{\mathcal F}_t\right)(\omega)  }{ \P \left(X^k_t=x^k |{\mathcal F}_t \right) (\omega) },
%nonumber
\\
 &&\quad \forall x^k,y^k\!\in\! {S}_k,\ y^k \neq x^k\ \textrm{and} \ \forall \omega \in B(t,k,x^k)\cap C(t,k,x^k),
\end{aligned}
\end{equation}
for almost every $\tT$.

\item[(ii)] The law of $X_0^k$ given $\mathcal{F}_T$ coincides with the law of $Y^k_0$ given $\mathcal{F}_T$.
\end{itemize}
\end{proposition}
\begin{proof}
Since $X$ is weakly Markovian consistent relative to $(\mathbb{F}, \cY)$, then, $X$ is weakly Markovian consistent relative to $(X^k,\mathbb{F})$ for each $k.$ Thus, in view of \eqref{vn-2} and Lemma \ref{lem:paskudny} we conclude that \eqref{vn-2a} holds. This proves (i). The conclusion (ii) is clear by the weak Markovian consistency of $X$ relative to $(\mathbb{F}, \cY)$.
\finproof
\end{proof}

%
%$\lambda^{k;x^kx^k}$ given by
%\[
%    \lambda^{k;x^kx^k}_t = - \sum_{y^k \in {S}_k, y^k \neq x^k } \lambda^{k;x^ky^k}_t,
%    \qquad
%    \forall x^k \in {S}_k
%    .
%\]
%
%
%
%for some  $\FF$-adapted processes $\lambda^{k;x^ky^k}.$

\begin{remark}\label{hftb}
 Even though the above proposition gives a necessary, rather than a sufficient, condition for the weak Markovian consistency of $X$ relative to $(\mathbb{F}, \cY)$, it will be skillfully used in construction of weak CMC copulae, in Section \ref{sec:wcmc}. In the present time we do not have a workable sufficient condition for  the weak Markovian consistency of $X$ relative to $(\mathbb{F}, \cY)$ to hold. Thus, for the time being, our strategy for constructing CMC copulae will be to use the necessary condition \eqref{vn-2a} to construct process $X$ which is a candidate for a CMC copula, and then to verify that this process indeed furnishes a weak CMC copula. We refer to Section   \ref{sec:wcmc} for details.
 \end{remark}

\subsection{When does weak Markov consistency imply strong Markov consistency?}

It is clear that the strong Markovian consistency for $X$ implies the weak Markovian consistency for $X$. As it will be seen in Section \ref{ex:weaknotstron}, process $X$ may be weakly Markovian consistent relative to $(X^k, \mathbb{F})$, but may fail to satisfy the  strong Markovian consistency condition relative to $(X^k, \mathbb{F})$. The following result provides sufficient conditions under which the weak Markovian consistency of $X$ relative to $(X^k, \mathbb{F})$ implies the strong Markovian consistency relative to $(X^k, \mathbb{F})$ for process $X$.
\bt\label{thm:weak->strong-cons}
Assume that $X$ satisfies the weak Markovian consistency condition relative to  $(X^k, \mathbb{F})$.
If  $\FF \vee \FF^{X^k}$ is $\P$-immersed in $\FF \vee \FF^{X}$, then $X$ satisfies the strong Markovian consistency condition relative to  $(X^k, \mathbb{F})$.
\et
\begin{proof}
Suppose that $\FF \vee \FF^{X^k}$ is immersed in
$\FF \vee \FF^{X}$, and let $X^k$ be an $(\FF, \FF^{X^k})$-CMC.
For arbitrary $x^k_1,  \ldots, x^k_m\in S_k $ and $0 \leq t_1 \leq \ldots \leq t_m \leq T,$ let us define the set $A$
\[
A:= \set{ X^k_{t_m } =x^k_{m},  \ldots, X^k_{t_1 } =x^k_{1} }.
\]
Since  $X^k$ is an $(\FF, \FF^{X^k})$-CMC,  we have for $s\leq t_1$
\[
	\P( A | \F_s \vee \sigma(X^k_s) )
=
	\P( A | \F_s \vee \F^{X^k}_s )
=	\P( A | \F_s \vee \F^X_s ),
\]
where in the second equality we have used immersion of  $\FF \vee \FF^{X^k}$ in
$\FF \vee \FF^{X}$ (cf. Section 6.1.1 in Bielecki and Rutkowski \cite{BieRut2001}). Thus $X^k$ is an $(\FF, \FF^{X})$-CMC. \finproof
\end{proof}
\begin{remark}
We note that the above theorem states only a sufficient condition for the weak Markovian consistency of $X$ to imply the strong Markovian consistency of $X$ (relative to  $(X^k, \mathbb{F}))$. As it is shown in \cite[Theorem 1.17]{BieJakNie2013}, in case of  trivial filtration $\FF$, the condition that $\FF^{X^k}$ is immersed in $\FF^{X}$ is both sufficient and necessary for weak Markovian consistency of $X$ to imply the strong Markovian consistency of $X$ (relative to  $X^k$).
\end{remark}

\section{CMC copulae }\label{CMCC}

As mentioned in the Introduction, the objective of the theory and practice of Markov copulae for classical Markov chains was to construct a non-trivial family of multivariate Markov chains such that components of each chain in the family are Markov chains (in some relevant filtrations) with given laws. Here, our goal is to extend the theory of Markov copulae from the universe of classical (finite) Markov chains to the universe of (finite) conditional Markov chains. Accordingly, we now use the term CMC copulae. As it turns out such extension is not a trivial one. But, it is quite important both from the mathematical point of view and from the practical point of view.

We will first discuss the so called strong CMC copulae, and then we will study the concept of the weak CMC copulae. In practice, an important role is played by the so called weak only CMC copulae, that is weak CMC copulae that are not strong CMC copulae (see discussion in \cite[Remark 2.3]{BieJakNie2013}). An example of such CMC copula will be given in Section \ref{ex:weaknotstron}.

We recall that in this paper the state space $S$ of process $X=(X^1,\ldots,X^N)$ is given as the Cartesian product $S_1\times S_2\times \ldots \times S_N.$

\subsection{ Strong  CMC copulae }\label{SCMCC}

\bd Let $\cY=\set{Y^1, \ldots, Y^N}$ be a family of processes, defined on some underlying probability space $(\Omega,\cA,\Q)$, such that each $Y^k $ is an $(\mathbb{F},\mathbb{F}^{Y^k})$-CMC with values in $S_k$. A \textit{strong CMC copula} between processes $Y^1, \ldots, Y^N$  is any multivariate process  $X=(X^1,\ldots,X^N)$, given on $(\Omega,\cA)$ endowed with some probability measure  $\P$, such that $X$ is an $(\FF,\FF^X)$--CMC, and such that it satisfies the strong Markovian consistency property with respect to  $(\mathbb{F},\cY)$.
\ed

The methodology developed in \cite{BieJakNie2014a}  allows us to construct strong CMC copulae between processes $Y^1, \ldots, Y^N$, that are defined on some underlying probability space $(\Omega,\cA,\Q)$ endowed with a reference filtration $\FF$, and are such that each $Y^k$ is $(\mathbb{F},\mathbb{F}^{Y^k})$-CDMC  with $\mathbb{F}$--intensity, say, $\Psi^k=[\psi^{k;x^ky^k}]_{x^k,y^k\in S_k}$. The additional feature of our construction is that, typically, the constructed CMC copulae $X$  are also $(\mathbb{F},\mathbb{F}^{X})$-DSMC.

In view of \cite[Theorem 3.4]{BieJakNie2014a}, Proposition \ref{cor:56} and Lemma \ref{lem:paskudny} a natural starting point for constructing a strong copula between $Y^1, \ldots, Y^N$ is to determine  a system of stochastic processes
$[\lambda^{xy}]_{x,y \in S}$ and an $S$-valued random variable $\xi=(\xi^1, \ldots, \xi^N)$ on $(\Omega,\cA)$,  such that they satisfy the following conditions:
\begin{description}
\item[(CMC-1)]
\begin{align*}
%\label{eq:copula-LSEa}
\psi^{k;x^ky^k}_t=\sum_{\substack{y^n\in S_n,\\ n=1,2,\ldots,N, n \neq k}}\lambda^{(x^1,\ldots,x^k,\ldots,x^N)(y^1,\ldots,y^k,\ldots,y^N)}_t,
\quad  \begin{array}{l}
x^n \in S_n, n=1,\ldots,N, \, \\
y^k \in S_k, y^k\neq x^k, \\
 k=1, \ldots, N, \tT.
\end{array}
%\
%\nonumber &
\end{align*}
%in unknowns $\lambda^{(x^1,\ldots,x^k,\ldots,x^N)(y^1,\ldots,y^k,\ldots,y^N)}_t$
%	\item[(CMC-2)] \[ \lambda^{xx}_t=- \sum_{z \in S :\, z \ne x} \lambda^{xz}_t,\quad x\in S,\ \tT,\]
%\begin{color}{red}
%\item[(CMC-3)]
%The matrix process $\Lambda_t=[\lambda^{xy}_t]_{x,y  \in S} $ correctly defines the $\FF$-intensity of a stochastic process, say $X$, with values in $S\, ,$ which is  a $(\mathbb{F},  \mathbb{F}^X)$-CMC and a $(\mathbb{F},  \mathbb{F}^X)$-DSMC.\end{color}
\item[(CMC-2)]
The matrix process $\Lambda_t=[\lambda^{xy}_t]_{x,y  \in S} $ satisfies canonical conditions relative to the pair $(S,\FF)$ (cf. \cite[Definition 3.3]{BieJakNie2014a}%\ref{cancon}
).
\item[(CMC-3)]
\[
    {\bQ}( \xi =y | \F_T ) = {\bQ}( \xi =y | \F_0 ), \quad \forall y \in S.
\]
\item[(CMC-4)]
\[
    {\bQ}( \xi^k =y^k | \F_T ) = \Q( Y^k_0 =y^k | \F_T ), \quad \forall y^k \in S_k, k =1, \ldots, N.
\]
\end{description}

\noindent
We will call any pair $(\Lambda,\xi)$ satisfying conditions (CMC-1)--(CMC-4) \emph{strong CMC pre-copula} between processes $Y^1, \ldots, Y^N$. Given a strong CMC pre-copula between processes $Y^1, \ldots, Y^N$ we can construct on  $(\Omega,\cA)$ a probability measure $\P$ and a process $X$,  using  
\cite[Theorem 3.4]{BieJakNie2014a} and starting from measure $\mathbb{Q}$ as above\footnote{It is always tacitly assumed that the probability space $(\Omega,\cA,\Q)$ is sufficiently rich so to support all stochastic processes and random variables that are considered throughout.}, such that $X$ is an $(\FF,\FF^X)$--CMC under $\bP$, and which satisfies the strong Markovian consistency property with respect to  $(\mathbb{F},\cY)$, in view of Proposition \ref{cor:56} and Lemma \ref{lem:paskudny}.
 Thus, it is a strong CMC copula between processes $Y^1, \ldots, Y^N$.
\begin{remark}\label{rem:PQ}
It follows from (3.9) and (3.10) in  \cite{BieJakNie2014a} that  for $\bP$ constructed as above we have
\begin{align*}
 &   {\bP}( \xi =y | \F_T ) = {\bP}( \xi =y | \F_0 ), \quad \forall y \in S. \\
&
    {\bP}( \xi^k =y^k | \F_T ) = \Q( Y^k_0 =y^k | \F_T ), \quad \forall y^k \in S_k, k =1, \ldots, N.
\end{align*}
\end{remark}
%Now, if a $\F_T$-conditional distribution of $X^k_0$ coincides with $\F_T$-conditional distribution of $Y^k_0$, for $k \in 1, \ldots, N$, then $X$ is strongly Markovian consistent with respect to $( \FF, Y )$ by \cite[Proposition 4.5]{BieJakNie2014a}%\ref{wn2.4}
%, and hence $X$ is a strong CMC copula between $Y^1, \ldots, Y^N$.
\brem\label{fikusna} (i) Note that in the definition of strong CMC copula it is required that $\F_T$-conditional distribution of $X^k_0$ coincides with $\F_T$-conditional distribution of $Y^k_0$, for $k \in 1, \ldots, N$, but, the $\F_T$-conditional distribution of the multivariate random variable $X_0=(X^1_0,\ldots,X^N_0)$ can be arbitrary. Thus, in principle, a  {strong CMC copula} between processes $Y^1, \ldots, Y^N$ can be constructed with help of a {strong CMC pre-copula} between processes $Y^1, \ldots, Y^N$, as well as a copula between the $\F_T$-conditional distributions of $X^k_0$s, for $k \in 1, \ldots, N$. For instance, in Example \ref{CIC} below, we take the components $X^1,\ldots,X^N$ are {conditionally independent given ${\mathcal{F}}_T$}.\\
(ii) In general, there exist numerous systems of stochastic processes that satisfy conditions (CMC-1) and (CMC-2), so that there exist numerous strong pre-copulae between conditional Markov chains $Y^1$, \ldots, $Y^N$, and, consequently, there exists numerous strong CMC copulae between conditional Markov chains $Y^1$, \ldots, $Y^N$. This is an important feature in financial applications, as it allows to calibrate a CMC model to both marginal data and to the basket data.
\erem

%Won do buczka:
%\mj{Czy kazda silna kopulka da sie tak zrobic ? Potrzebne twierdzenie ktore trzeba by udowodnic: jezeli X jest \emph{strongly Markovian consistent}  wzgledem $(X^k,\FF)$ (lub $(\FF, Y)$) z intensywnoscia $\Lambda$ to istnieje $\FF$ intensywnosc $\Gamma$ ktora spelnia (ASM-k).
%\\
%Problem tzw. "proper $\FF$-intensity" - jedyna wyznaczajaca rozklad intensywnosc ! TO EWENTUALNIE  DO NASTEPNEJ PRACY !! \tr{To znaczy, ze tu nie jest nam potrzebne, tak?}
% }

Below we provide examples of strong CMC copulae. The first example, dealing with conditionally independent univariate CMCs, does not really address the issue of modeling dependence between components of a multivariate CMC. Nevertheless, on one hand, this example may have non-trivial practical applications in insurance, and, on the other hand, it  is a non-trivial example from the mathematical point of view.
Moreover, this example provides a sort of a reality check for the theory of strong CMC copulae: it would be not good for the theory if a multivariate conditional Markov chain $X=(X^1,\ldots,X^N)$ with conditionally independent components were not a strong CMC  copula.

\subsection{Examples}\label{examples}

\subsubsection{Conditionally independent strong CMC copula}\label{CIC} %\textbf{()}

 This example generalizes the independent Markov copula example presented in Section 2.1 in \cite{BieJakNie2013}.

Let $Y^1$, \ldots, $Y^N$ be processes such that each $Y^k $ is an $(\mathbb{F},\mathbb{F}^{Y^k})$-CDMC with values in $S_k$, and with $\FF$--intensity $\Psi^k_t=[\psi^{k;x^ky^k}_t]_{x^k,y^k \in S_k}$. Assume that for each $k$ the process $\Psi^k$ satisfies canonical conditions relative to the pair $(S_k,\FF)$ . Additionally assume that
\be\label{eq:cosik-init-y}
\Q(Y^k_0 =x^k | \F_T ) =\Q(Y^k_0 =x^k | \F_0 ), \quad \forall x^k \in S_k, \ k=1, \ldots, N.
\ee

 Consider a matrix valued random process $\Lambda$  given as the following Kronecker sum
\be\label{eq:LLL}
	\Lambda_t = \sum_{k=1}^N
	I_1 \otimes \ldots {\otimes} I_{k-1} {\otimes} \Psi^k_t {\otimes} I_{k+1} {\otimes} \ldots {\otimes} I_N
	%I^{\set{k}^c} \otimes \Psi^k_t
, \quad \tT,
\ee
%where we use
%\[
%I^{\set{k}^c} \otimes \Psi^k_t :=  I_1 \otimes \ldots {\otimes} I_{k-1} {\otimes} \Psi^k_t {\otimes} I_{k+1} {\otimes} \ldots {\otimes} I_N,
%\]
%and
where $\otimes$ is the Kronecker product (see e.g. Horn and Johnson \cite{HorJoh1994}), and where $I_k$ denotes the identity matrix of dimensions $|S_k| \times |S_k|$.
Moreover, let us take an $S$-valued random variable $\xi=(\xi^1, \ldots , \xi^N )$,
which has $\F_T$-conditionally independent coordinates, that is
\begin{equation}\label{eq:xi-cond-ind}
{\bQ}(\xi^1 =x^1, \ldots, \xi^N = x^N | \F_T )
=
\prod_{i=1}^N \bQ(\xi^i = x^i | \F_T ) , \quad
\forall x=(x^1, \ldots, x^N) \in S.
\end{equation}
Additionally assume that $\F_T$-conditional distributions of coordinates of $\xi$ and $Y_0$ coincide, meaning that
\be\label{eq:cosik-init-xiy}
\bQ(\xi^k =x^k | \F_T ) =\Q(Y^k_0 =x^k | \F_T ), \quad \forall x^k \in S_k, \ k=1, \ldots, N.
\ee
%meaning that $\F_T$-conditional distributions of coordinates of $\xi$ and $Y_0$ coincide.
As it  is shown in Appendix B, Proposition \ref{lem:cudolemma},  $\Lambda$ satisfies conditions (CMC-1) and (CMC-2). Furthermore, by
\eqref{eq:xi-cond-ind} and  \eqref{eq:cosik-init-y}, $\xi$ satisfies (CMC-3) and, by \eqref{eq:cosik-init-xiy}, also (CMC-4).
Thus, $(\Lambda, \xi)$ is a strong CMC pre-copula between conditional Markov chains $Y^1$, \ldots, $Y^N$.
Now, as we said earlier, we can construct, with the help of \cite[Theorem 3.4]{BieJakNie2014a}%\ref{thm:CMC-under-constr}
, a multivariate $(\mathbb{F},\mathbb{F}^{X})$-CDMC (see also \cite[Proposition 4.17]{BieJakNie2014a}%\ref{thm:CMC-DCMC}
), say $X=(X^1,\ldots,X^N)$, with values in $S$, which, in view of Proposition \ref{cor:56} satisfies the strong Markovian consistency property with respect to  $(\mathbb{F},\cY)$.
Therefore, the process $X$ furnishes a {strong CMC copula} between processes $Y^1, \ldots, Y^N$.
%Since $\Lambda$ satisfies Condition (CMC-1), process $X$ furnishes a strong CMC copula between conditional Markov chains $Y^1$, \ldots, $Y^N$.
Finally, Proposition \ref{cor:cond-ind} of Appendix B demonstrates  that components of $X$ are conditionally independent given ${\mathcal{F}}_T$.

It is quite clear from \eqref{eq:LLL} that components $X^i$ of $X$ do not jump simultaneously; this, indeed, is the inherent feature of the conditional independent CMC copula. In the next example we will present a strong CMC copula such that its components  have common jumps.
\subsubsection{Common jump strong CMC copula}\label{silaczek}
%\begin{example}\textbf{(Common jump strong CMC copula)}

Let us consider two processes, $Y^1$ and $Y^2$, such that each $Y^i$
is an $(\FF,\FF^{Y^i})$-CDMC taking values in the state space $\{0,1\}.$
Suppose that their $\FF$-intensities are
\begin{align*}
\Psi^1(t)=
\bordermatrix{ &0 &1 \cr
   0&  -a_t  & a_t \cr
   1&  0 & 0
},\qquad
\Psi^2(t)
=\bordermatrix{ &0 &1 \cr
0 &     -b_t  & b_t \cr
1 &     0 & 0 },
\end{align*}
where $a,b$ are nonnegative $\FF$-progressively measurable stochastic processes which have countably many jumps.
%Solutions to Kolmogorov equations can be obtained easily
%\[
%P^1(t,u) =
%\left(
%  \begin{array}{cc}
%    e^{ - \int_t^u a_s + c_s ds} & 1- e^{ - \int_t^u a_s + c_s ds} \\
%    0 & 1 \\
%  \end{array}
%\right)
%\]
%\[
%P^2(t,u) =
%\left(
%  \begin{array}{cc}
%    e^{ - \int_t^u b_s + c_s ds} & 1- e^{ - \int_t^u b_s + c_s ds} \\
%    0 & 1 \\
%  \end{array}
%\right)
%\]
Moreover assume that
\[
	\Q( Y^1_0 = 0 )= \Q( Y^2_0 = 0 ) = 1.
\]
Next, let $\Lambda$ be a matrix valued process given by
\[
	\Lambda_t =
\bordermatrix{&(0,0)&(0,1)&(1,0)&(1,1)\cr
  (0,0) &  -(a_t + b_t - c_t) & b_t - c_t  & a_t - c_t & c_t \cr
  (0,1) & 0 & -a_t & 0 & a_t  \cr
  (1,0) & 0 & 0 & -b_t & b_t \cr
  (1,1) & 0 & 0 & 0 & 0 },
\]
where $c$ is an  $\FF$-progressively measurable stochastic processes, which has countably many jumps and such that \[
0 \leq c_t \leq a_t \wedge b_t
,\quad \tT.
\]
Moreover, let $\xi$ be an $S$-valued random variable satisfying
\[
	{\text{Prob}}( \xi = (0,0) ) = 1.
\]
 It can be easily checked that  $(\Lambda,  \xi )$ satisfies conditions (CMC-1)-(CMC-4), so that it is a strong CMC pre-copula between conditional Markov chains $Y^1$, $Y^2$.
 Now, in view of \cite[Theorem 3.4]{BieJakNie2014a} %\ref{thm:CMC-under-constr}
and \cite[Proposition  4.17]{BieJakNie2014a}%\ref{thm:CMC-DCMC}
, one can construct a stochastic process  $X=(X^1,X^2)$, which is a two-variate $(\FF, \FF^X)$-CDMC with an $\FF$-intensity $\Lambda$ and such that $X_0 = \xi$. Moreover, by Proposition \ref{cor:56}, the process $X$ is strongly Markovian consistent with respect to $(\FF,\cY)$ and hence $X$ is a strong CMC copula between $Y^1$ and $Y^2$. { Note also that, in  view of interpretation of intensity, the coordinates of the process $X$ have common jumps, provided that $c > 0$. }

\begin{remark}
%'One may argue that the example belows is too simple to be applied in practice.
We have chosen this very simple example just to illustrate an idea of construction of strong copulae for CMC. One can, in similar spirit as in \cite{BCCH2014}, generalize it to arbitrary dimension $N$ preserving that each marginal process is two-states absorbing CMC. Then the $\FF$-intensity matrix  has a similar structure as in the above example, i.e. its entries are marginal intensities minus intensities of "common jumps" to absorbing states.
Generalization to a higher number of non-absorbing states is tricky and requires clever parametrization, since number of free parameters in strong CMC copula becomes enormously large (see e.g. \cite{BieVidVid2008}).
\end{remark}

\subsubsection{Perfect dependence strong CMC copula}\label{PDC}
%It was observed in Section \ref{niezakoniecznie} that the fact that
%Condition (ASM-k) holds for all  $k=1,2,\ldots,N$, does not
%imply that strong Markovian consistency holds for $X$. Thus, it should be possible to construct a strong CMC copula  $X$ such that the corresponding intensity process $\Lambda$ does not satisfy Condition (ASM-k) for some $k$. The present example shows that this indeed is the case.

Let $Y^1$, \ldots, $Y^N$ be processes such that each $Y^k $ is an $(\mathbb{F},\mathbb{F}^{Y^k})$-CMC, and such that they have the same $\cF_T$ conditional laws. Consider process $X=(X^1,\ldots,X^N)$, where $X^k=Y^1,\quad k=1,2,\ldots,N$. It is clear that $X$ furnishes a strong CMC copula between conditional Markov chains $Y^1$, \ldots, $Y^N$.

Obviously other CMC copulae between $Y^1, \ldots, Y^N$, such as conditional independence copulae, can be constructed.

\subsection{Weak CMC Copulae }\label{sec:wcmc}

\bd Let $\cY=\set{Y^1, \ldots, Y^N}$ be a family of processes, {defined on some underlying probability space $(\Omega,\cA,\Q)$,} such that each $Y^k $ is an $(\mathbb{F},\mathbb{F}^{Y^k})$-CMC with values in $S_k$. A \textit{weak CMC copula} between processes $Y^1, \ldots, Y^N$ is any multivariate process  $X=(X^1,\ldots,X^N)$, defined on $(\Omega,\cA)$ endowed with some probability measure $\P$, such that $X$ is an $(\FF,\FF^X)$--CMC, and such that it satisfies the weak Markovian consistency property with respect to  $(\mathbb{F},\cY)$.
\ed

Similarly as in the case of the strong CMC copulae, the methodology developed in \cite{BieJakNie2014a} allows us to construct weak CMC copulae between processes $Y^1, \ldots, Y^N$, that are defined on some underlying probability space $(\Omega,\cA,\Q)$ endowed with a reference filtration $\FF$, and are such that each $Y^k$ is $(\mathbb{F},\mathbb{F}^{Y^k})$-CDMC  with $\mathbb{F}$--intensity, say, $\Psi^k=[\psi^{k;x^ky^k}]_{x^k,y^k\in S_k}$.
%The additional feature of our construction is that, typically, the constructed CMC copulae $X$  are also $(\mathbb{F},\mathbb{F}^{X})$-DSMC.

In view of \cite[Theorem 3.4]{BieJakNie2014a}, Proposition \ref{prop:wmcnec},  Lemma \ref{lem:paskudny}, as well as of Remark \ref{hftb},  a natural starting point for constructing a weak CMC copula ${X}$  between $Y^1, \ldots, Y^N$ is to determine  any system of stochastic processes
$(\lambda^{xy})_{x,y \in S}$ and any $S$-valued random variable $\xi=(\xi^1, \ldots, \xi^N)$ on $(\Omega,\cA)$ and to find a probability measure $\bP$, such that the following conditions are satisfied:

\begin{description}
%in unknowns $\lambda^{(x^1,\ldots,x^k,\ldots,x^N)(y^1,\ldots,y^k,\ldots,y^N)}_t$
%	\item[(CMC-2)] \[ \lambda^{xx}_t=- \sum_{z \in S :\, z \ne x} \lambda^{xz}_t,\quad x\in S,\ \tT,\]
%\begin{color}{red}
%\item[(CMC-3)]
%The matrix process $\Lambda_t=[\lambda^{xy}_t]_{x,y  \in S} $ correctly defines the $\FF$-intensity of a stochastic process, say $X$, with values in $S\, ,$ which is  a $(\mathbb{F},  \mathbb{F}^X)$-CMC and a $(\mathbb{F},  \mathbb{F}^X)$-DSMC.\end{color}
\item[(WCMC-1)]
The matrix process $\Lambda_t=[\lambda^{xy}_t]_{x,y  \in S} $ satisfies canonical conditions relative to the pair $(S,\FF).$
\item[(WCMC-2)]
\[
    {\bP}( \xi =y | \F_T ) = {\bP}( \xi =y | \F_0 ), \quad \forall y \in S.
\]
\item[(WCMC-3)]
\[{\bP}( \xi^k =y^k | \F_T ) = \Q( Y^k_0 =y^k | \F_T ), \quad \forall y^k \in S_k, k =1, \ldots, N.
\]
\item[(WCMC-4)]
\begin{equation}
\begin{aligned}\lab{vn-2.1}
\!\!\!\!\!\!\!\!\!\!\!\!\!\!\!\!\!\!\!\!
 \psi^{k;x^ky^k}_t(\omega)
&
\!=\! \sum_{\substack{x^n,y^n\in {S}_n \\n=1,2,\ldots,N, n \neq k }}  & \!\!\!\!\!\!\!\lambda^{(x^1, \ldots, x^N)(y^1, \ldots, y^N)}_t(\omega)
\frac{{\P} \left(X^1_t=x^1, \ldots, X^N_t=x^N|{\mathcal F}_t\right)(\omega)  }{{\P} \left(X^k_t=x^k |{\mathcal F}_t \right) (\omega) },
%nonumber
\\
 && \!\!\!\!\!\!\!\forall x^k,y^k\!\in\! {S}_k,\ y^k \neq x^k\ \textrm{and} \ \forall \omega \in B(t,k,x^k)\cap C(t,k,x^k),
\end{aligned}
\end{equation}
for almost every $\tT$, where process $X=(X^1,\ldots,X^N)$ is a $(\FF,\FF^X)$--CMC {(under probability measure $\bP$)} with intensity $\Lambda$ and initial distribution given by $\xi.$
\end{description}

We will call any triple $(\Lambda,\xi,X)$ satisfying conditions (WCMC-1)--(WCMC4) \emph{a base for weak CMC copula} between processes $Y^1, \ldots, Y^N$, and we will call the process $X$ in $(\Lambda,\xi,X)$ \emph{a candidate for weak CMC copula} between processes $Y^1, \ldots, Y^N$. So, a possible method for constructing a weak CMC copula between processes $Y^1, \ldots, Y^N$ is to first construct a  base $(\Lambda,\xi,X)$ for weak CMC copula between processes $Y^1, \ldots, Y^N$, and then
 to skillfully verify that the candidate process $X$ satisfies the weak Markovian consistency property with respect to  $(\mathbb{F},\cY)$, and thus, that it is a {weak CMC copula} between processes $Y^1, \ldots, Y^N$. We will illustrate application of this method in Section \ref{ex:weaknotstron}.

Before we proceed to the next subsection, we observe that remark analogous to Remark \ref{fikusna} applies in the case of the weak CMC copulae.

Finally, it needs to be said that in several applications an important role is played by the so called \emph{weak only CMC copulae}, that is, the weak CMC copulae that are not strong CMC copulae. The next section provides an example of a weak only CMC copula. %\tr{moze dopisac o contagion}

\subsubsection{Example of a weak CMC copula that is not strong CMC copula}\label{ex:weaknotstron}
%\begin{example}\label{ex:weaknotstron}

In Section \ref{examples} we gave three examples of strong CMC copulae. Consequently, they are also examples of weak CMC copulae. Here, we will give an example of a weak only CMC copula.
In particular, this property implies that in the present example the immersion property formulated in Theorem  \ref{thm:weak->strong-cons} is not satisfied.

Let us consider processes $Y^1$ and $Y^2$, { defined on some probability space $(\Omega,\cA,\Q)$,} such that each $Y^i$ is an $(\FF, \FF^{Y^i})$-CDMC
taking values in the state space $S_i=\set{0,1}$. We assume that $\FF$-intensities of $Y^1$ and $Y^2$ are, respectively,
\begin{align*}
\Psi^1_t =
\left(
\begin{array}{cc}
	-(a_t  + c_t) + c_t  \frac{ \alpha_t }{  \delta_t  + \alpha_t } & (a_t  + c_t) - c_t  \frac{ \alpha_t }{  \delta_t  + \alpha_t }\\
0 & 0
\end{array}
\right),
\end{align*}
\begin{align*}
\Psi^2_t =
\left(
\begin{array}{cc}
	-(b_t  + c_t) + c_t  \frac{ \beta_t }{  \delta_t + \beta_t } & (b_t  + c_t) - c_t  \frac{ \beta_t }{  \delta_t  + \beta_t }\\
0 & 0
\end{array}
\right),
\end{align*}
where
\[
    \alpha_t = e^{ - \int_0^t a_u du} \int_0^t b_u e^{-\int_0^u (b_v + c_v) dv} du, \qquad
    \beta_t = e^{ - \int_0^t b_u du} \int_0^t a_u e^{-\int_0^u (a_v + c_v) dv} du,
\]
\[
\delta_t = e^{-\int_0^t (a_u + b_u + c_u) du},
\]
for $a,b,c$ being positive  $\FF$-adapted stochastic processes.
Moreover, suppose that
\[
\Q( Y_0^i = 0  )
=
1
\quad i = 1, 2,
\]
which implies $\Q( Y_0^i = 0 | \cF_T) = 1 $.

Our goal is to find a weak CMC copula between $Y^1$ and $Y^2$.
Towards this end we will look for an $(\FF, \FF^X)$-CMC process $X$,
{ defined on some probability space $(\Omega,\cA,\P)$,} satisfying condition \eqref{vn-2.1} adapted to the present setup. In particular, the state space of process $X$ needs to be equal to $S = \set{ (0,0),(0,1),(1,0),(1,1)}$.

However, since condition \eqref{vn-2.1} is a necessary condition for weak Markovian consistency with respect to $(\FF, \cY)$,  but not a sufficient one in general, then a process satisfying \eqref{vn-2.1} may not be weakly Markovian consistent with respect to $(\FF, \cY)$.
Nevertheless, we will construct an $(\FF,\FF^X)$-CMC process $X$ that satisfies condition \eqref{vn-2.1} and is weakly Markovian consistent with respect to $(\FF, \cY)$, so that it is a weak CMC copula between $Y^1$ and $Y^2$.

Let us consider stochastic process $X$ with state space $S$, which is an $(\FF, \FF^X)$-CDMC with an $\FF$-intensity matrix $\Lambda$ given by
\begin{align}
\Lambda_t=
\bordermatrix{&(0,0)&(0,1)&(1,0)&(1,1)\cr
                (0,0)&-(a_t+b_t+c_t) &  b_t  & a_t & c_t\cr
                (0,1)& 0  &  -a_t & 0 & a_t \cr
                (1,0)& 0 & 0 & -b_t & b_t\cr
                (1,1)& 0  &   0 &0 & 0}
%=
%\left(
%\begin{array}{cccc}
%	-(a_t + b_t + c_t) & b_t & a_t & c_t \\
%		0	& -a_t & 0 & a_t \\
%		0  & 0 & -b_t & b_t \\
%0 & 0 & 0 & 0
%\end{array}
%\right)
,
\end{align}
and with the initial distribution
\[
    {\P }( X_0 =(0,0)) = 1.
\]
% It turns out that a solution to the necessary condition \eqref{vn-2} is a
%valid $\FF$
The components  $X^1$ and $X^2$ are processes with state space $S = \set{0,1}$, and such that the state $1$ is an absorbing state for both $X^1$ and $X^2$. Thus, by similar arguments as in \cite[Example 2.4]{BieJakNie2015}, $X^1$ (resp. $X^2$) is an $(\FF, \FF^{X^1})$-CDMC (resp. $(\FF, \FF^{X^2})$-CDMC). Consequently, $X$ is a weakly Markovian consistent process relative to $(X^1, \FF)$ ($(X^2, \FF)$ resp.).
%This can be easily shown using similar arguments as in \cite[Example 2.4]{BieJakNie2015}.

We will now  compute, using \eqref{vn-2}, an $\FF$-intensity of $X^1$ and an $\FF$-intensity of  $X^2$.
To this end we first solve the conditional Kolmogorov forward equation for $P(s,t)=[p_{xy}(s,t)]_{x,y \in S}$, i.e.
\begin{align}\label{eq:cond-Kol}
d P(s,t ) = P(s,t) \Lambda(t) dt, \quad  P(s,s) = \II.
\end{align}
This is done in order to compute the following conditional probabilities (see \cite[Theorem 4.10]{BieJakNie2014a})
\[
	\P(X_t = y | \F_T \vee \F^X_s) = \sum_{x \in S } \I_\set{X_s =x} p_{xy}(s,t),
\]
which will be used in computation of conditional probabilities, of the form
%\sout{(cf. \eqref{vn-2})}
\[
    \frac{\P(X^1_t = x^1, X^2_t = x^2  | \F_t )}{ \P(X^1_t = x^1  | \F_t )}.
\]
One can easily verify, by solving appropriate ODEs, that the unique solution of \eqref{eq:cond-Kol} is given by
\[
P(s,t)=
\left(
  \begin{array}{cccc}
    e^{ - \int_s^t (a_u + b_u + c_u) du }  & \alpha(s,t) & \beta(s,t)&\gamma(s,t)\\
    0 & e^{ - \int_s^t a_u du } & 0 & 1 - e^{ - \int_s^t a_u du } \\
    0 & 0 & e^{ - \int_s^t b_u  du} & 1 - e^{ - \int_s^t b_u du}\\
    0 & 0 & 0 & 1 \\
  \end{array}
\right),
\]
where
\begin{align*}
\alpha(s,t) &= e^{ - \int_s^t a_u du} \int_s^t b_u e^{-\int_s^u (b_v + c_v) dv} du,
\\
\beta(s,t) &= e^{ - \int_s^t b_u du} \int_s^t a_u e^{-\int_s^u (a_v + c_v) dv} du,
\\
\gamma(s,t) &= 1- e^{ - \int_s^t (a_u + b_u + c_u) du } -\alpha(s,t) - \beta(s,t).
\end{align*}
%Now we will use necessary condition \eqref{vn-2} to compute its intensities.
Since $X$ is started at $(0,0)$, then by \cite[Proposition 4.6]{BieJakNie2014a}
we have that
\begin{align*}
	\P(X^1_t = 0, X^2_t = 0 | \F_t) &=
\E (\P(X^1_t = 0, X^2_t = 0 | \F_T)| \F_t)
= \E (e^{-\int_0^t (a_u + b_u + c_u) du} | \F_t)
\\
&= e^{-\int_0^t (a_u + b_u + c_u) du} = \delta_t.
\end{align*}
In an analogous way we conclude that
\begin{align*}
	\P(X^1_t = 1, X^2_t = 0 | \F_t) &= \beta (0,t), \\
	\P(X^1_t = 0, X^2_t = 1 | \F_t) &= \alpha (0,t).
\end{align*}
Thus
\begin{align*}
	\frac{\P(X^1_t = 1, X^2_t = 0 | \F_t) }{ \P(X^2_t = 0 | \F_t)}
&= \frac{ \beta_t }{  \delta_t  + \beta_t },
\\
	\frac{\P(X^1_t = 0, X^2_t = 1 | \F_t) }{ \P(X^1_t = 0 | \F_t)}
&= \frac{ \alpha_t }{  \delta_t  + \alpha_t },
\end{align*}
{where for brevity of notation we let
$\alpha_t = \alpha (0,t), \beta_t= \beta (0,t)$.
}
Here \eqref{vn-2} takes the form
\begin{align*}
\1_\set{ X^1_t=  0} \lambda^{1;01}_t
&=
\1_\set{ X^1_t=  0}\bigg(
( \lambda^{(00)(10)}_t
+ \lambda^{(00)(11)}_t)
\frac{\P(X^1_t = 0, X^2_t = 0 | \F_t)}{\P( X^1_t = 0| \F_t)}
\\
& \qquad  \qquad +
(\lambda^{(01)(10)}_t
+
\lambda^{(01)(11)}_t )
\frac{\P(X^1_t = 0, X^2_t = 1 | \F_t)}{\P( X^1_t = 0| \F_t)}
\bigg)
 \\
&=
\1_\set{ X^1_t=  0}\bigg(( a_t + c_t)
\frac{\P(X^1_t = 0, X^2_t = 0 | \F_t)}{\P( X^1_t = 0| \F_t)}
+
a_t
\frac{\P(X^1_t = 0, X^2_t = 1 | \F_t)}{\P( X^1_t = 0| \F_t)}
\bigg)
 \\
 &  =
    \1_\set{ X^1_t=  0}\bigg( (a_t  + c_t) - c_t  \frac{ \alpha (0,t) }{  \delta(0,t)  + \alpha (0,t) }
    \bigg),
\end{align*}
and
\begin{align*}
\1_\set{ X^1_t=  1} \lambda^{1;10}_t
&=
\1_\set{ X^1_t=  1}\bigg(
( \lambda^{(10)(00)}_t
+ \lambda^{(10)(01)}_t)
\frac{\P(X^1_t = 1, X^2_t = 0 | \F_t)}{\P( X^1_t = 1| \F_t)}
\\
& \qquad  \qquad +
(\lambda^{(11)(00)}_t
+
\lambda^{(11)(01)}_t )
\frac{\P(X^1_t = 1, X^2_t = 1 | \F_t)}{\P( X^1_t = 1| \F_t)}
\bigg)
 \\
&=
\1_\set{ X^1_t=  1}\bigg( 0
\frac{\P(X^1_t = 0, X^2_t = 0 | \F_t)}{\P( X^1_t = 0| \F_t)}
+
0
\frac{\P(X^1_t = 0, X^2_t = 1 | \F_t)}{\P( X^1_t = 0| \F_t)}
\bigg)
 \\
 &  = \1_\set{ X^1_t=  1} 0.
\end{align*}
Therefore, since $\P(X^1_t=0| \F_t) > 0$, the $\FF$-intensity  of $X^1$ is given by
\begin{align*}
\Lambda^1(t) =
\left(
\begin{array}{cc}
	-(a_t  + c_t) + c_t  \frac{ \alpha_t }{  \delta_t  + \alpha_t } & (a_t  + c_t) - c_t  \frac{ \alpha_t }{  \delta_t  + \alpha_t }\\
0 & 0
\end{array}
\right)
=
\Psi^1_t
.
\end{align*}
Analogously, the $\FF$-intensity  of $X^2$ is given by
\begin{align*}
\Lambda^2(t) =
\left(
\begin{array}{cc}
	-(b_t  + c_t) + c_t  \frac{ \beta_t }{  \delta_t + \beta_t } & (b_t  + c_t) - c_t  \frac{ \beta_t }{  \delta_t  + \beta_t }\\
0 & 0
\end{array}
\right)=
\Psi^2_t.
\end{align*}
Consequently $X$ is a weak CMC copula for $Y^1 $ and $Y^2$.

Now we will demonstrate that $X$ is in fact weak only CMC copula for $Y^1 $ and $Y^2$.
We have
\begin{align*}
    &\P( X^1_t = 0 | \F_T \vee \sigma(X_s) )
    \1_\set{ X^1_s = 0,  X^2_s = 0 }
    =
    \1_\set{ X^1_s = 0,  X^2_s = 0 }
    (p_{(0,0)(0,0)}(s,t)
    +
    p_{(0,0)(0,1)}(s,t)
    )
    \\
    &=
    \1_\set{ X^1_s = 0,  X^2_s = 0 }
    \bigg(
    e^{ - \int_s^t (a_u + b_u + c_u) du }
    +
    %\alpha(s,t)
e^{ - \int_s^t a_u du} \int_s^t b_u e^{-\int_s^u (b_v + c_v) dv} du
    \bigg),
\end{align*}
and
\begin{align*}
    \P( X^1_t = 0 | \F_T \vee \sigma(X_s) )
    \1_\set{ X^1_s = 0,  X^2_s = 1 }
    &=
    \1_\set{ X^1_s = 0,  X^2_s = 1 }
    (p_{(0,1)(0,0)}(s,t)	
    +
    p_{(0,1)(0,1)}(s,t)
    )
    \\
    &=
    \1_\set{ X^1_s = 0,  X^2_s = 1 }
    e^{ - \int_s^t a_u  du }.
\end{align*}
Clearly
\begin{equation}\label{eq:weak-only}
\bigg(
    e^{ - \int_s^t (a_u + b_u + c_u) du }
    +
    %\alpha(s,t)
e^{ - \int_s^t a_u du} \int_s^t b_u e^{-\int_s^u (b_v + c_v) dv} du
    \bigg)
    \neq
e^{ - \int_s^t a_u  du },
\end{equation}
 unless $c\equiv 0$ on $[s,t]$. In this case
\eqref{eq:weak-only} implies that
\[
\P\big(
\P( X^1_t = 0 | \F_s \vee \F^X_s )
\neq
\P( X^1_t = 0 | \F_t \vee \sigma(X^1_s) )
\big) > 0.
\]
Thus process $X$ is not strongly Markovian consistent, so $X$ is  a weak only CMC copula between $Y^1$ and $Y^2$ unless $c\equiv 0$.
  Nevertheless,  process $X$ is a strong CMC copula between $Y^1$ and $Y^2$ for $c\equiv 0$, which follows from  Section \ref{silaczek}.

%\end{example}
\begin{remark}
Note that $\Lambda(t)$ admits the following representation
\begin{equation}\label{rozwiniatko}
	\Lambda(t) = \Psi^1(t) \otimes I_2 + I_1 \otimes \Psi^2(t)
+ B_{12}(t) - B_1(t) - B_2(t),
\end{equation}
where
\begin{align*}
&\Psi^1(t) \otimes I_2 + I_1 \otimes \Psi^2(t) \\
&\!\!\!\!\!\!\!\!\!\!\!\!\!\!=
\left(
\begin{array}{cccc}
	-(a_t + b_t + c_t  \frac{ \delta_t }{ \delta_t + \beta_t } + c_t  \frac{ \delta_t }{ \delta_t + \alpha_t } ) & b_t + c_t  \frac{ \delta_t }{ \delta_t + \beta_t } & a_t + c_t  \frac{ \delta_t }{  \delta_t  + \alpha_t } & 0 \\
		0	& -a_t - c_t  \frac{ \delta_t }{  \delta_t  + \alpha_t } & 0 & a_t + c_t  \frac{ \delta_t }{  \delta_t  + \alpha_t } \\
		0  & 0 & -b_t - c_t  \frac{ \delta_t  }{  \delta_t  + \beta_t } & b_t + c_t  \frac{ \delta_t  }{  \delta_t  + \beta_t } \\
0 & 0 & 0 & 0
\end{array}
\right).
\end{align*}
gives the conditionally independent copula between $Y^1$ and $Y^2$ (cf. Section \ref{CIC}), and the remaining terms

\begin{align*}
B_{12}(t) &= \left(
\begin{array}{cccc}
	-c_t & 0 & 0 & c_t \\
		0	& 0 & 0 & 0 \\
		0  & 0 & 0 & 0 \\
0 & 0 & 0 & 0
\end{array}
\right), \\
B_1(t) &= \left(
\begin{array}{cccc}
	-c_t  \frac{ \delta_t  }{  \delta_t + \beta_t } & c_t  \frac{ \delta_t  }{  \delta_t  + \beta_t } & 0 & 0 \\
		0	& 0 & 0 & 0 \\
		0  & 0 & -c_t  \frac{ \delta_t  }{  \delta_t  + \beta_t }  & c_t  \frac{ \delta_t  }{  \delta_t  + \beta_t }  \\
0 & 0 & 0 & 0
\end{array}
\right), \\ %\quad
B_2(t) &= \left(
\begin{array}{cccc}
	-c_t  \frac{ \delta_t }{  \delta_t  + \alpha_t }  & 0 & c_t  \frac{ \delta_t}{  \delta_t  + \alpha_t }  & 0 \\
		0	& -c_t  \frac{ \delta_t }{  \delta_t  + \alpha_t }  & 0 & c_t  \frac{ \delta_t }{  \delta_t  + \alpha_t }  \\
		0  & 0 & 0 & 0 \\
0 & 0 & 0 & 0
\end{array}
\right)
\end{align*}
introduce the dependence structure between $Y^1$ and $Y^2$.

Representations of the form \eqref{rozwiniatko} are important for construction of CMC copulae and will be studied in detail in  Bielecki, Jakubowski and Niew\k{e}g{\l}owski \cite{BJN-buczek}.
\end{remark}

\section{Applications to  the premium evaluation for unemployment insurance products}\label{BGW}

In the recent paper by Biagini, Groll and  Widenmann \cite{BiaWid2013} a very interesting problem of evaluation  of  premia  for unemployment insurance products, for a pool of individuals, was considered. We would like to suggest here a possible generalization of the model studied in \cite{BiaWid2013}; this generalization, we believe, may provide a more adequate way to deal with computation of the premia.

\subsection{Model of Biagini, Groll and Widenmann }
Biagini et al. \cite{BiaWid2013} used the DSMC framework to model the dynamics of employment status of an individual. The dynamics are modeled in \cite{BiaWid2013} under the probability measure, say $\P,$ called a real-world measure. Then, using these dynamics they aim at computing for $\tT$ the insurance premium, which is denoted as $P_t$ .

In \cite{BiaWid2013}, the evolution of the employment status of an individual $k$ is given in terms of a Markov chain, say $X^k$, which takes values in the state space $S_k=\{1,2\},$ where the state $"1``$ indicates that the individual is employed, and the $"2``$ indicates that the individual is unemployed. It is assumed that process $X^k$ is an $(\FF^Z,\FF^{X^k})$-DSMC, where $\FF^Z$ is a reference filtration generated by some factor process $Z$.

As stated earlier, the quantity to be computed for the individual $k$ is the value of the premium of insurance against unemployment. Roughly speaking, the premium $P^k_t$ at time $t$ is given as
\[P_t^k=\E_\P(\Phi_k(X^k)|\G^k_t),\]
where $\Phi_k$ is some random functional of process $X^k$, and where $\G^k_t=\F^Z_t\vee \F^{X^k}_t.$ In particular, the premium at time $t=0$ needs to be computed, that is
\[P_0^k=\E_\P(\Phi_k(X^k)|\G^k_0).\]
Note, that we have written $P_0^k$ as a conditional expectation, given  $\G^k_0$, rather than the unconditional expectation, as it is done in formula (2) in \cite{BiaWid2013}.

\subsection{Proposed CMC copula approach}
We think that, for the purpose of evaluation of premia for unemployment insurance products for a pool of individuals, labeled as $k=1,2,\ldots,N,$
it is important to account for possible dependence between processes $X^k,\ k=1,2,\ldots,N.$

Thus, we think that it may be advantageous to enrich the model studied in \cite{BiaWid2013} by considering a process~$Y=(Y^1,\ldots,Y^N),$ which is a  CMC copula between processes $X^k,\ k=1,2,\ldots,N.$

Thanks to copula property, the characteristics of dependence between processes~$X^k,\ k=1,2,\ldots,N$ can be estimated separately from estimation of the distributional characteristics of each process $X^k.$ The latter task can be efficiently executed using the methodology outlined in  \cite{BiaWid2013}.

The premium $P^k_t$ at time $t$ is given in the CMC copula model as
\[P_t^k=\E(\Phi_k(Y^k)|\widehat \G^{k}_t),\]
where  $\widehat \G^k_t=\F^Z_t\vee \F^{Y^k}_t.$ {If process $Y$ is constructed as a weak only CMC copula between processes $X^k,\ k=1,2,\ldots,N,$ then we have that, with $\widehat \G_t=\F^Z_t\vee \F^{Y}_t,$ 
\[\E(\Phi_k(Y^k)|\widehat \G^k_t) \neq \E(\Phi_k(Y^k)|\widehat \G_t).\]}
This, of course, means that the employment status of the entire pool is influences the calculation of the individual premium, a feature, which we think is important.

{The theory of strong and weak CMC copulae can be extended in straightforward manner to modeling structured dependence between subgroups of processes $X^k,\ k=1,2,\ldots,N,.$ This will allow for study of insurance premia modeling for relvant subgroups of employees.}

\section*{Appendix A}

%The following lemma will be used below.
\begin{lemma}\label{lem:paskudny}
Let  $Z$ be an $(\mathbb{F},\mathbb{F}^{Z})$-CDMC and let $U$ be an $(\mathbb{F},\mathbb{F}^{U})$-CDMC, with values in some (finite) state space $\wh S$, and with  intensities $\Gamma^Z$  and $\Gamma^U$,  respectively.
Then, the conditional law of $Z$ given $\mathcal{F}_T$ coincides with the conditional law of $U$ given $\mathcal{F}_T$ if and only if
%\begin{enumerate}
%\item
\begin{align}
\lab{eq:intZequalsU}
 \Gamma^Z&=\Gamma^U \quad d u \otimes d \P-a.e.,  \\
\lab{eq:lawZ0equalsU0} \P(Z_0 =x | \F_T ) &= \P(U_0 =x | \F_T ) \quad \forall x \in \wh{S}.\end{align}
%
%%\item
%conditional law of $Z_0$ given $\mathcal{F}_T$ coincides with the conditional law of $U_0$ given $\mathcal{F}_T$.
%%\end{enumerate}
\end{lemma}
\begin{proof}
First we prove sufficiency.
So, suppose that   \eqref{eq:intZequalsU} and \eqref{eq:lawZ0equalsU0} hold. Recall that the c-transition fields $P^Z$ ($P^U$ respectively) satisfy Kolmogorov equations
\eqref{eq:INT-trans-prob-backward} and \eqref{eq:INT-trans-prob-forward}.
% for all  $v\leq t$,
%\begin{equation}\label{eq:INT-trans-prob-backward}
% P^Z(v,t) - \mathrm{I}     =   \int_v^t  \Gamma^Z_u P^Z(u,t) du,
%\end{equation}
%\begin{equation}\label{eq:INT-trans-prob-forward}
% P^Z(v,t)  - \mathrm{I} = \int_v^t P^Z(v,u) \Gamma^Z_u du.
%\end{equation}
%%with $\Gamma^Z$ beeing an $\FF$-intensity ($\Gamma^U$ respectively).
Since \eqref{eq:intZequalsU} holds we have, by uniqueness of solutions of Kolmogorov equations, that $P^Z=P^U$.
%
%\[
%	P^Z(v,t) - P^U(v,t) = \int_v^t P^Z(v,u)\Gamma^Z_u - P^U(v,u) \Gamma^U_u du=
%\int_v^t (P^Z(v,u)- P^U(v,u)) \Gamma^U_u du
%\]
This and \eqref{eq:lawZ0equalsU0}, by \cite[Proposition 4.6]{BieJakNie2014a} %\ref{wn2.4}
 (see eq. (4.8)%\eqref{eq:c-fidis}
) imply that
conditional law of $Z$ and $U$ given $\F_T$ coincide.
\\
\noindent
Now we prove necessity.
Suppose that conditional laws of $Z$ and $U$ given $\F_T$  coincide, we want to show  that \eqref{eq:intZequalsU} and \eqref{eq:lawZ0equalsU0}  hold.
First, note that the equality of conditional laws of $Z$ and $U$ given $\F_T$  implies \eqref{eq:lawZ0equalsU0}.
To show that \eqref{eq:intZequalsU} holds it suffices to show that their c-transition fields are equal.  Indeed, this equality implies that for any $0 \leq v \leq t \leq T$
\[
	0 = P^Z(v,t) - P^U(v,t) = \int_v^t \left( P^Z(v,u)\Gamma^Z - P^U(v,u) \Gamma^U_u \right) du = \int_v^t P^Z(v,u)(\Gamma^Z_u - \Gamma^U_u) du.
\]
%thus for any $t \geq v \geq 0$ it holds
Consequently
\[
P^Z(v,u)(\Gamma^Z_u - \Gamma^U_u) = 0, \quad du \otimes d \P-a.e.
 \ (on \ [v,T]),
\]
since $P^Z(v,u)$ is invertible (cf. \cite[Proposition 3.11]{JakNie2010}). This in turn implies
\eqref{eq:intZequalsU}.
%Now we will show that c-transition fields of $Z$ and $U$ are equal.
%Note that by assumption, for any
%$s \leq t$, $x,y \in S $,  we have that
%\be\lab{eq:Ndim-cond}
%	\P(Z_t  = y, Z_s =x | \F_T )
%=
%	\P(U_t  = y, U_s =x | \F_T )
%\ee
%and that
%\be\lab{eq:1dim-cond}
%\P( Z_s =x | \F_T )
%=
%\P( U_s =x | \F_T ).\ee
%Thus
%\begin{align*}
%	\P(Z_t  = y, Z_s =x | \F_T )
%&
%%=\E ( \I_\set{Z_t = y, Z_s = x} | \F_T )
%=\E ( \E ( \I_\set{Z_t = y, Z_s = x}| \F_T \vee \G_s) | \F_T )
%\\
%&=
%\E ( \I_\set{ Z_s = x}\E ( \I_\set{Z_t = y}| \F_T \vee \G_s) | \F_T )
%=
%\E ( \I_\set{ Z_s = x}p^Z_{xy}(s,t) | \F_T )
%\\
%&=
%\P( Z_s = x | \F_T ) p^Z_{xy}(s,t).
%\end{align*}
%The tower property of conditional expectations and  $\F_t$-measurability of $p^Z_{xy}(s,t)$ yields
%\[
%\P(Z_t  = y, Z_s =x | \F_t )
%=\P( Z_s = x | \F_t ) p^Z_{xy}(s,t).
%\]
%Since analogous formula is valid for $U$, we have that \eqref{eq:Ndim-cond} implies
%\[
%\P( Z_s = x | \F_t ) p^Z_{xy}(s,t)
%=
%\P( U_s = x | \F_t ) p^U_{xy}(s,t)
%\]
%thus by \eqref{eq:1dim-cond}, the tower property of conditional expectations  and  \eqref{tp}  we have that
%\[
%p^Z_{xy}(s,t)=
%p^U_{xy}(s,t).
%\]
This ends the proof.
\finproof
\end{proof}

\section*{Appendix B: Discussion of conditional independence copula}

This appendix provides some technical results needed in Section \ref{CIC}. In this section we denote by $I$ the identity matrix of dimension $|S|$, and by $I_k$ the identity matrix of dimension $|S_k|.$

\begin{proposition}\label{lem:cudolemma}
Suppose that we are given an $N$-tuple of matrix valued processes $\Psi^k_t=[\psi^{k;xy}_t]_{x,y  \in S_k} $, $k=1, \ldots, N$, which  satisfy canonical conditions relative to the pair $(S_k,\FF)$ for every $k=1, \ldots, N$. Then the process $\Lambda$ given in \eqref{eq:LLL} satisfies conditions (CMC-1) and (CMC-2).
\end{proposition}
\begin{proof}
In what follows, we will use a convention  that for $A \subset \wt S,$ where $\wt S$ is a finite set, the characteristic function
\[
\I_A(j) =
\begin{cases}
1 & \mbox{if}\ j \in A,\\
0 & \mbox{if}\ j \notin A,
\end{cases}
\]
will be interpreted as a vector in $\mathbf{R}^{|\wt S|}$, written as $\I^{\wt S}_A;$ for simplicity, we will also denote  $\I_{\wt S}=\I^{\wt S}_{\wt S}.$

Let us fix $k \in \set{ 1, \ldots, N}$, $x^k , y^k \in S_k$, and $\bar{x} = (\bar{x}^1, \ldots, \bar{x}^N) \in S$ such that $\bar{x}^k = x^k$.
  Now, referring to (CMC-1), we observe that
\[\sum_{\substack{y^n\in S_n,\\ n=1,2,\ldots,N, n \neq k}}\lambda^{(\bar{x}^1,\ldots,\bar{x}^k,\ldots,\bar{x}^N)(y^1,\ldots,y^k,\ldots,y^N)}_t=\big (\I^{S}_\set{\bar{x}}\big )^\top \Lambda_t \mathbf{v}^{y^k},\]
where
\[
\mathbf{v}^{y^k} := \I_{S_1} \otimes \ldots \otimes  \I_{S_{k-1}} \otimes \I^{S_k}_\set{y^k} \otimes  \ldots \otimes  \I_{S_{N}}.
\]

%if $m=k$ then we have
%\begin{align*}
%&(I_1 \otimes \ldots {\otimes} I_{k-1} {\otimes} \Lambda^k_t {\otimes} I_{k+1} {\otimes} \ldots {\otimes} I_n)(\I_{E_1} \otimes \ldots \otimes  \I_{E_{k-1}} \otimes \I_\set{j_k} \otimes  \ldots \otimes  \I_{E_{N}}
%	)\\
%&=(I_1 \I_{E_1}) \otimes \ldots {\otimes}
%(I_{k-1} \I_{E_{k-1}}) {\otimes} (\Lambda^k_t\I_\set{j_k} ) {\otimes}
%(I_{k+1} \I_{E_{k+1}})
%{\otimes} \ldots {\otimes}
%(I_n \I_{E_{N}}) \\
%&=\I_{E_1} \otimes \ldots {\otimes}
% \I_{E_{k-1}} {\otimes} (\Lambda^k_t\I_\set{j_k} ) {\otimes}
%\I_{E_{k+1}}
%{\otimes} \ldots {\otimes}
%\I_{E_{N}} \end{align*}
%%\mn{w inny sposob}
%Since the matrix $\Lambda^{(N)}_t$ is a sum of some matrices we need to multiply
%each summand by $\mathbf{v}^{y^k}$ separately and then add  the result.

Next, we see that
\[\Lambda_t \mathbf{v}^{y^k}=\sum_{m=1}^N\, \Phi^m_t,\]
where
\begin{align*}
&
\Phi^k_t=\left(
(\otimes_{p=1}^{k-1} I_p) \otimes  \Psi^k_t  \otimes (\otimes_{q=k+1}^N I_n)  \right)\left(
(\otimes_{p=1}^{k-1} \I_{S_p}) \otimes  \I^{S_k}_\set{y^k}  \otimes (\otimes_{q=k+1}^N  \I_{S_q})   \right)
\\
&
=\left(
(\otimes_{p=1}^{k-1} I_p \I_{S_p}) \otimes \Psi^k_t   \I^{S_k}_\set{y^k}  \otimes (\otimes_{q=k+1}^N I_q \I_{S_q})   \right)
=\left(
(\otimes_{p=1}^{k-1} \I_{S_p}) \otimes \Psi^k_t   \I^{S_k}_\set{y^k}  \otimes (\otimes_{q=k+1}^N \I_{S_q})   \right).
 \end{align*}
We have, for $m>k$,
\begin{align*}
&
\Phi^m_t=\left(
(\otimes_{p=1}^{m-1} I_p) \otimes  \Psi^m_t  \otimes (\otimes_{q=m+1}^N I_n)  \right)\left(
(\otimes_{p=1}^{k-1} \I_{S_p}) \otimes  \I^{S_k}_\set{y^k}  \otimes (\otimes_{q=k+1}^N  \I_{S_q})   \right)
\\
&
=\left(
(\otimes_{p=1}^{k-1} I_p \I_{S_p})
\otimes (I_p \I^{S_k}_\set{y^k})
\otimes (\otimes_{r=k}^{m-1} I_p \I_{S_p})
 \otimes (\Psi^m_t   \I_{S_m} ) \otimes (\otimes_{q=m+1}^N I_q \I_{S_q})   \right)
\\
&
=\left(
(\otimes_{p=1}^{k-1} \I_{S_p})
\otimes (\I^{S_k}_\set{y^k})
\otimes (\otimes_{r=k}^{m-1} \I_{S_p})
 \otimes (0_{S_m} ) \otimes (\otimes_{q=m+1}^N  \I_{S_q})   \right)
=\otimes_{p=1}^N 0_{S_p}
=0_{\times_{p=1}^N {S_p} }
 \end{align*}
and, for $m<k$,
\begin{align*}
&
\Phi^m_t=\left(
(\otimes_{p=1}^{m-1} I_p) \otimes  \Psi^m_t  \otimes (\otimes_{q=m+1}^N I_n)  \right)\left(
(\otimes_{p=1}^{k-1} \I_{S_p}) \otimes  \I^{S_k}_\set{y^k}  \otimes (\otimes_{q=k+1}^N  \I_{S_q})   \right)
\\
&
=\left(
(\otimes_{p=1}^{m-1} I_p \I_{S_p})
 \otimes (\Psi^m_t   \I_{S_m} )
\otimes (\otimes_{r=m+1}^{k-1} I_r \I_{S_r})
\otimes (I_k \I^{S_k}_\set{y^k})
 (\otimes_{q=k+1}^N I_q \I_{S_q})   \right)
\\
&
=\left(
(\otimes_{p=1}^{m-1} \I_{S_p})
 \otimes (0_{S_m} )
\otimes (\otimes_{r=m+1}^{k-1} \I_{S_r})
\otimes ( \I^{S_k}_\set{y^k})
 (\otimes_{q=k+1}^N \I_{S_q})   \right)
=\otimes_{p=1}^N 0_{S_p}
=0_{\times_{p=1}^N {S_p} }.
 \end{align*}
Consequently,  for any $\bar{x} = (\bar{x}^1, \ldots, \bar{x}^N) \in S$ such that $ \bar{x}^{k} = x^k$ and $y^k \in S_k$,  we have  that
\begin{align*}
\big (\I^{S}_\set{\bar{x}}\big )^\top \Lambda_t \mathbf{v}^{y^k}
&=
\sum_{m=1}^N\, \big (\I^{S}_\set{\bar{x}}\big )^\top  \Phi^m_t
=
\big (\I^{S}_\set{\bar{x}}\big )^\top  \Phi^k_t
\\
&=
\Big(
\prod_{p=1}^{k-1} \I_{S_p}(\bar{x}^p) \Big) \Psi^k_t   \I^{S_k}_\set{y^k}(\bar{x}^k)
\Big(\prod_{q=k+1}^N \I_{S_q}(\bar{x}^q)   \Big)
=\psi^{k;x^ky^k}_t.
\end{align*}
% in the $\bar{x}$-row of this vector we have
%\[
%\left(
%(\otimes_{p=1}^{k-1} \I_{S_p}) \otimes \Psi^k_t   \I^{S_k}_\set{y^k}  \otimes (\otimes_{q=k+1}^N \I_{S_q})   \right)_{(\bar{x}^1, \ldots, \bar{x}^N)}
%=
%\left(
%\prod_{p=1}^{k-1} \I_{S_p}(\bar{x}^p) \right) \Psi^k_t   \I^{S_k}_\set{y^k}(\bar{x}^k)
%\left(\prod_{q=k+1}^N \I_{S_q}(\bar{x}^q)   \right)
%=\psi^{k;x^ky^k}_t
%\]%
%Thus for any $\bar{x} = (\bar{x}^1, \ldots, \bar{x}^N) \in S$ such that $ \bar{x}^{k} = x^k$, $y^k \in S_k$ the sum on the right hand side of (CMC-1) is
%\[
%	(\Lambda^{(N)}_t)
%\left(
%(\otimes_{p=1}^{k-1} \I_{S_p}) \otimes  \I^{S_k}_\set{y^k}  \otimes (\otimes_{q=k+1}^N  \I_{S_q})   \right)_{(\bar{x}^1, \ldots, \bar{x}^N)}
%=
%\psi^{k;x^ky^k}_t.
%\]
This proves that  $\Lambda$ satisfies (CMC-1).

The fact that $\Lambda$ satisfies (CMC-2) follows from the assumption that
 $\Psi^k_t=[\psi^{k;xy}_t]_{x,y  \in S} $, satisfy canonical conditions relative to the pair $(S_k,\FF)$ for every $k=1, \ldots, N$,
and from the following representation of the entries of $\Lambda_t$:
\[
\lambda^{(x^1, \ldots, x^N)(y^1, \ldots,y^N)}_t
%=
%\sum_{m=1}^N
%\I_\set{ y^n =x^n, \forall n \in \set{1,\ldots,N} \setminus \set{ m } } \psi^{m; x^my^m}_t
=
\sum_{m=1}^N
\Big(\prod_{\substack{n=1 \\ n \neq m}}^N\I_\set{ y^n =x^n} \Big) \psi^{m; x^my^m}_t.
\]
\finproof
\end{proof}

\begin{lemma}\label{lem:811}
Suppose that we are given an $N$-tuple of matrix valued processes $\Psi^k_t=[\psi^{k;xy}_t]_{x,y  \in S_k} $, $k=1, \ldots, N$, which  satisfy canonical conditions relative to the pair $(S_k,\FF)$,
and are such that
\be\label{eq:psi-int}
\sum_{x^k \in S_k } \int_0^T \abs{\psi^{k; x^k,x^k}_s} ds < \infty,
\quad
\forall
k=1, \ldots, N.
\ee
Let us fix $s\in [0,T]$, and let $P^{(N)}(s,\cdot)$ be the solution of
\be\label{eq:Pn-0}
d P^{(N)}(s,t) = P^{(N)}(s,t) \Lambda^{(N)}_t dt,
\quad P^{(N)}(s,s) = I,\quad t\in [s,T],
\ee
with $\Lambda^{(N)}=\Lambda$, where $\Lambda$ is given in \eqref{eq:LLL}.
Then, we have that
\begin{align}\label{eq:PN}
P^{(N)}(s,t) = \bigotimes_{k=1}^N P_k(s,t)
\end{align}
where, for $k=1, \ldots,N$,
\[
d P_k(s,t) = P_k(s,t) \Psi^{k}_t dt, \quad P_k(s,s) = I_{k},\quad t\in [s,T].
\]
\end{lemma}
\begin{proof}
We will verify that $P^{(N)}$ defined by \eqref{eq:PN} satisfies \eqref{eq:Pn-0}, this by uniqueness of solutions of \eqref{eq:Pn-0} will imply desired result.
We will proceed by induction on %$n\in \set{1,\dots,N}.$
$N$. First, we  take $N=2$ and we prove that $P^{(2)}(s,\cdot)$  given as
\[
	P^{(2)}(s,t) := P_1(s,t) {\otimes} P_2(s,t),
\]
satisfies \eqref{eq:Pn-0} which takes  the form
\be\label{eq:Q}
	d P^{(2)}(s,t)=P^{(2)}(s,t)(\Psi^1_t \otimes I_2  + I_1 \otimes \Psi^2_t ) dt, \quad P^{(2)}(s,s) = I.
\ee
 By the mixed-product rule (cf. \cite[Lemma 4.2.10]{HorJoh1994}) we can write  $P^{(2)}(s,t)$ as
\be\label{eq:Q=P1P2}
P^{(2)}(s,t)
=
(P_1(s,t) I_1 )\otimes (I_2 P_2(s,t))
=
Q_1(s,t) Q_2(s,t),
\ee
where
\[
Q_1(s,t)
= P_1(s,t) \otimes  I_2 , \quad
Q_2(s,t)=
  I_1 \otimes P_2(s,t)  .
\]
Thus, to show \eqref{eq:Q} we need to prove that
\[
d(Q_1(s,t)Q_2(s,t)) = (Q_1(s,t)Q_2(s,t)) (\Psi^1_t \otimes I_2  + I_1 \otimes \Psi^2_t )dt .
\]
%One can show that we have
%\[
%dQ_1(s,t)	=Q_1(s,t) (\Psi^1_t \otimes I_2) dt
%, \quad
%dQ_2(s,t)	=Q_2(s,t) (I_1 \otimes \Psi^2_t) dt.
%\]
Since,
\begin{align*}
d Q_1(s,t) &= d (P_1(s,t) \otimes  I_2 ) = (d P_1(s,t)) \otimes  I_2
=
( P_1(s,t) \Psi^1_t  dt) \otimes  I_2
\\
&=( P_1(s,t) \otimes  I_2  ) (\Psi^1_t  \otimes  I_2 )dt
=Q_1(s,t)(\Psi^1_t  \otimes  I_2 )dt.
\end{align*}
and, similarly,
\[dQ_2(s,t)	=Q_2(s,t) (I_1 \otimes \Psi^2_t) dt,\]
then, using integration by parts we get
\begin{align*}
d(Q_1(s,t)Q_2(s,t))
&=
(dQ_1(s,t))Q_2(s,t) + Q_1(s,t)dQ_2(s,t) \\
&=
Q_1(s,t) (\Psi^1_t \otimes I_2)Q_2(s,t)dt + Q_1(s,t)Q_2(s,t) (I_1 \otimes \Psi^2_t) dt
\\
&=Q_1(s,t) Q_2(s,t)(\Psi^1_t \otimes I_2)dt + Q_1(s,t)Q_2(s,t) (I_1 \otimes \Psi^2_t) dt
\\
&=Q_1(s,t) Q_2(s,t)(\Psi^1_t \otimes I_2+I_1 \otimes \Psi^2_t) dt ,
\end{align*}
where the third equality  follows since $Q_2(s,t)$ and $(\Psi^1_t \otimes I_2)$ commute.
Indeed, definition of $Q_2$ and the mixed-product property imply
\[
Q_2(s,t)(\Psi^1_t \otimes I_2) = (I_1 \otimes P_2(s,t))(\Psi^1_t \otimes I_2) = (I_1 \Psi^1_t) \otimes (P_2(s,t)I_2) = \Psi^1_t \otimes P_2(s,t)
\]
and analogously
\[
(\Psi^1_t \otimes I_2) Q_2(s,t) = (\Psi^1_t \otimes I_2) (I_1 \otimes P_2(s,t))
=
(\Psi^1_t I_1) \otimes  (I_2P_2(s,t)) = \Psi^1_t \otimes P_2(s,t).
\]
This demonstrates that $P^{(2)}(s,\cdot)$ satisfies \eqref{eq:Q}. Consequently, in view of the uniqueness of the solution of \eqref{eq:Q}, the result of the lemma is proved in case $N=2.$

Now, let us assume that the result of the lemma holds for some $N \geq 2$.
We want to show that
\[
P^{(N+1)}(s,t) := \bigotimes_{k=1}^{N+1} P_k(s,t)%) \otimes  P_k(s,t)
\]
satisfies
\[
d P^{(N+1)}(s,t) = P^{(N+1)}(s,t) \Lambda^{(N+1)}_t dt,
\]
where
\[
\Lambda^{(N+1)}_t := \sum_{k=1}^{N+1} I_1 \otimes \ldots {\otimes} I_{k-1} {\otimes} \Psi^k_t {\otimes} I_{k+1} {\otimes} \ldots {\otimes} I_{N+1}.
\]
Towards this end we note that
\begin{align*}
P^{(N+1)}(s,t) &= P^{(N)}(s,t)  \otimes  P_{N+1}(s,t)
=
(P^{(N)}(s,t) I^{(N)}) \otimes (I_{N+1} P_{N+1}(s,t)) \\
&=
(P^{(N)}(s,t) \otimes I_{N+1} ) (I^{(N)} \otimes  P_{N+1}(s,t)),
\end{align*}
where the third equality follows from the mixed product rule, and where
\[
	I^{(N)} := \bigotimes_{k=1}^N I_k.
\]
Now, we have that
\begin{align*}
d(P^{(N)}(s,t) \otimes I_{N+1} )
%=
%((P^{(N)}(s,t) \Lambda^{(N)}_t) \otimes I_{N+1} )  dt
&=
(P^{(N)}(s,t) \otimes I_{N+1})(\Lambda^{(N)}_t \otimes I_{N+1} )  dt,\quad
P^{(N)}(s,s) \otimes I_{N+1}  = I^{(N+1)},
\end{align*}
and
\begin{align*}
d (I^{(N)} \otimes  P_{N+1}(s,t))
&= (I^{(N)} \otimes  P_{N+1}(s,t))( I^{(N)} \otimes  \Psi^{N+1}) dt,
\quad I^{(N)} \otimes  P_{N+1}(s,s) = I^{(N+1)} .
\end{align*}
Thus, since  matrices
$(\Lambda^{(N)}_t \otimes I_{N+1} )$  and $ (I^{(N)} \otimes  P_{N+1}(s,t)) $
commute, integration by parts yields
\[
d P^{(N+1)}(s,t)
=
P^{(N+1)}(s,t) (\Lambda^{(N)}_t  \otimes I_{N+1} +  (I_{1} {\otimes} \ldots {\otimes} I_{N}) \otimes \Psi^{N+1}_t) dt, \quad P^{(N+1)}(s,s)  = I^{(N+1)}.
\]
Since we have that
\[
\Lambda^{(N+1)}_t  = \Lambda^{(N)}_t  \otimes I_{N+1} +  (I_{1} {\otimes} \ldots {\otimes} I_{N}) \otimes \Psi^{N+1}_t,
\]
the proof is complete.
%since by induction hypothesis $P^{(n)}(s,t)$ satisfies \eqref{eq:Pn}
%the proof can be reduced to the proof for $n=2$.
%\mn{Czy napewno ?}
%.....
%\mn{wydaje sie ze to sa takie same obliczenia}
\finproof
\end{proof}
%\[
%P^1=\bordermatrix{ &0 &1 \cr
%   0&  p^1_{00} & p^1_{01} \cr
%   1&  p^1_{10}  & p^1_{11}
%}
%\quad
%P^2=\bordermatrix{ &0 &1 \cr
%   0&  p^2_{00} & p^2_{01} \cr
%   1&  p^2_{10}  & p^2_{11}
%}
%\]
%\[
%P^1 \otimes P^2 =
%\begin{pmatrix}
%	p^1_{00}
%\begin{pmatrix}
%	p^2_{00} & p^2_{01} \\
%   p^2_{10}  & p^2_{11}
%\end{pmatrix}
% & p^1_{01} \begin{pmatrix}
%	p^2_{00} & p^2_{01} \\
%   p^2_{10}  & p^2_{11}
%\end{pmatrix}\\
%   p^1_{10} \begin{pmatrix}
%	p^2_{00} & p^2_{01} \\
%   p^2_{10}  & p^2_{11}
%\end{pmatrix} & p^1_{11} \begin{pmatrix}
%	p^2_{00} & p^2_{01} \\
%   p^2_{10}  & p^2_{11}
%\end{pmatrix}
%\end{pmatrix}
%=
%\bordermatrix{ & (0,0) &  (0,1) &(1,0) &(1,1) \cr
%(0,0) & 	p^1_{00} p^2_{00} & p^1_{00} p^2_{01}  & p^1_{01} p^2_{00} & p^1_{01} p^2_{01} \cr
%(0,1) & 	p^1_{00}  p^2_{10}& p^1_{00} p^2_{11}  & p^1_{01} p^2_{10} & p^1_{01} p^2_{11} \cr
%(1,0) &    p^1_{10} p^2_{00} & p^1_{10} p^2_{01} & p^1_{11} p^2_{00} & p^1_{11}  p^2_{01} \cr
%(1,1) &    p^1_{10} p^2_{10} & p^1_{10} p^2_{11} & p^1_{11} p^2_{10}  & p^1_{11}  p^2_{11}
%}
%\]
%so that we have
%\[
%(P^1 \otimes P^2)_{(x^1x^2)(y^1y^2)}
%=
%p^{1}_{x^1y^1}p^{2}_{x^2y^2}
%\]

\begin{proposition}\label{prop:cond-indep-DSMC}
Suppose that $X=(X^1,\ldots, X^N)$ is an $S$-valued  $(\FF, \GG)$-DSMC with c-transition field of the form
\be\lab{eq:Ndim-Kron-prod}
P(s,t) = \bigotimes_{k=1}^N P_k(s,t),
\ee
where $P_k = [p_{k;xy}]_{x,y\in S_k}$ is a stochastic matrix valued  random field on $S_k$, for $k=1, \ldots N$. Moreover  assume that for all $x=(x^1, \ldots, x^N) \in S$ it holds
\be\lab{eq:N-in-dist-cond-ind}
\P \left(\bigcap_{k=1}^N \set{ X^k_{0} = x^k } \Big| \F_T \right)
= \prod_{m=1}^N \P \left( X^k_{0}=x^k  \Big| \F_T \right) \!\!.
\ee
Then, the components $X^1,\ldots, X^N$ of  $X$ are conditionally independent given $\F_T$.
\end{proposition}
\begin{proof}
It suffices to prove that for any $t_1 ,\ldots, t_n \in [0,T]$, and
for any sets $A^m_k \subset S_m$, $m=1, \ldots N$, $k=1, \ldots, n$  it holds that
\be\lab{eq:con-ind-fidis-N}
	\P\left( \bigcap_{m=1}^N \bigcap_{k=1}^n \set{ X^m_{t_k}\in A^m_k } \Big| \F_T \right)
	= \prod_{m=1}^N \P \left(\bigcap_{k=1}^n \set{ X^m_{t_k}\in A^m_k } \Big| \F_T \right) \!\!.
\ee

%\mn{+dyskusja oparta o argumenty z Kallenberga.
%
%Suppose that we have two processes $X^1$, $X^2$ such that its finite dimensional distributions satisfy
%\[
%\P	(X^1_{t_1} \in A_1, \ldots, X^1_{t_d}\in A_d, X^2_{s_1} \in B_k, \ldots, X^2_{s_k} \in B_k)
%=\P	(X^1_{t_1} \in A_1, \ldots, X^1_{t_d}\in A_d) \P (X^2_{s_1} \in B_k, \ldots, X^2_{s_k} \in B_k)
%%
%%	(X^1_{t_1}, \ldots, X^1_{t_d}) (X^2_{s_1}, \ldots, X^2_{s_k})
%\]
%We claim this implies that their laws are indepndent.
%
%Suppose that we have second pair of independent proceses $Y^1, Y^2$ with the laws equal $Y^1 = X^1$ , $Y^2 = X^2$. Then because of independence we have
%\[
%\P	(Y^1_{t_1} \in A_1, \ldots, Y^1_{t_d}\in A_d, Y^2_{s_1} \in B_k, \ldots, Y^2_{s_k} \in B_k)
%=\P	(Y^1_{t_1} \in A_1, \ldots, Y^1_{t_d}\in A_d) \P (Y^2_{s_1} \in B_k, \ldots, Y^2_{s_k} \in B_k)
%%
%%	(X^1_{t_1}, \ldots, X^1_{t_d}) (X^2_{s_1}, \ldots, X^2_{s_k})
%\]
%then because the laws are equal $Y^1 = X^1$ , $Y^2 = X^2$ we have that
%\[
%\P	(X^1_{t_1} \in A_1, \ldots, X^1_{t_d}\in A_d, X^2_{s_1} \in B_k, \ldots, X^2_{s_k} \in B_k)
%=
%\P	(Y^1_{t_1} \in A_1, \ldots, Y^1_{t_d}\in A_d, Y^2_{s_1} \in B_k, \ldots, Y^2_{s_k} \in B_k)
%\]
%In other words the finite dimensional laws of $X$ and $Y$ are equal. This implies that the law of processes $X$,$Y$ are equal and thus the law of $X$ is a product  of the law of  $X^1$ and $X^2$.
%}
%
For simplicity, we will give the proof of \eqref{eq:con-ind-fidis-N} for $N=2$. The proof in the general case proceeds along the same lines and will be omitted.  We prove \eqref{eq:con-ind-fidis-N} in steps.

%We will prove that for $X=(X^1, X^2)$ - $(\FF, \GG)$-DSMC with c-transition field
%\be\lab{eq:Ndim-Kron-prod}
%	P(s,t) = P^1(s,t) \otimes P^2(s,t),
%\ee
%satisfying
%\be\lab{eq:in-dist-cond-ind}
%	\P(X^1_0 =x^1, X^2_0 = x^2 | \F_T )
%=
%	\P(X^1_0 =x^1| \F_T ) 	\P( X^2_0 = x^2 | \F_T ), \quad \forall x=(x^1, x^2) \in S_1 \times S_2
%\ee
%we have that
%\be\lab{eq:con-ind-fidis}
%	\P\left( \bigcap_{k=1}^n \set{ X^1_{t_k}\in A^1_k }\cap \bigcap_{k=1}^n \set{  X^2_{t_k}\in A^2_k} | \F_T \right)
%	= \P \left(\bigcap_{k=1}^n \set{ X^1_{t_k}\in A^1_k } | \F_T \right) \P \left(\bigcap_{k=1}^n \set{X^2_{t_k}\in A^2_k }| \F_T \right)\!\!,
%\ee
%for any sets $A^1_1, \ldots, A^1_n \subset S_1$ and
%$A^2_1, \ldots, A^2_n \subset S_2$.
%

\noindent \underline{Step 1:} Let us first note that \eqref{eq:Ndim-Kron-prod} and definition of the Kronecker product imply that for any $(x^1,x^2), \ (y^1,y^2) \in S_1 \times S_2$ we have that
\begin{equation}\label{eq:tran-mac}
p_{(x^1,x^2)(y^1,y^2)}(s,t)
=
 p_{1;x^1y^1}(s,t)p_{2;x^2y^2}(s,t).
\end{equation}
In addition, as we will show now, if $P_1(s,t)$ and $P_2(s,t)$ satisfy $\F_T$-conditional Chapmann-Kolmogorov equation (cf. \cite[Theorem 3.6]{JakNie2010}), then $(P(s,t))_{0\leq s\leq t \leq T}$ defined by \eqref{eq:Ndim-Kron-prod} satisfies $\F_T$-conditional  Chapmann-Kolmogorov equation as well. Indeed, applying the mixed-product rule to the right hand side of \eqref{eq:Ndim-Kron-prod}  we obtain
\begin{align*}
P(s,t) P(t,u)
&=
(P_1(s,t) \otimes P_2(s,t))(P_1(t,u) \otimes P_2(t,u))
\\
&=
(P_1(s,t)P_1(t,u)) \otimes (P_2(s,t)P_2(t,u))
=
P_1(s,u) \otimes P_2(s,u)=P(s,u).
\end{align*}
\noindent \underline{Step 2:} We will show that $X^1$ and $X^2$ are $(\FF, \GG)$-DSMC with c-transition fields $P_1$ and $P_2$.
Towards this end, we first observe that
\begin{align*}
& \P(X^1_t =y^1| \F_T \vee \G_s) \I_\set{ X^1_s = x^1, X^2_{s} = x^2} =
\I_\set{ X^1_s = x^1, X^2_{s} = x^2} \sum_{y^2 \in S_2} \P(X^1_t =y^1, X^2_t = y^2| \F_T \vee \G_s)\\
& =\I_\set{ X^1_s = x^1, X^2_{s} = x^2} \sum_{y^2 \in S_2}  p_{1;x^1y^1}(s,t)p_{2;x^2y^2}(s,t)
=\I_\set{ X^1_s = x^1, X^2_{s} = x^2} p_{1;x^1y^1}(s,t)
\left(\sum_{y^2 \in S_2}  p_{2;x^2y^2}(s,t)\right)
\\
&=\I_\set{ X^1_s = x^1, X^2_{s} = x^2} p_{1;x^1y^1}(s,t),
\end{align*}
where the second equality follows from \eqref{eq:tran-mac}.
Now, summing this equality over $x^2 \in S_2$ yields
\[
 \P(X^1_t =y^1| \F_T \vee \G_s) \I_\set{ X^1_s = x^1}
= \I_\set{ X^1_s = x^1} p_{1;x^1y^1}(s,t),
\]
which means that $X^1$ is an $(\FF,\GG)$-DSMC with c-transition field $P_1$. Analogously we can prove that $X^2$ is an $(\FF,\GG)$-DSMC with c-transition field $P_2$.

\noindent \underline{Step 3:}  Now, we will prove that \eqref{eq:con-ind-fidis-N} holds.

Towards this end we will first restate  \eqref{eq:con-ind-fidis-N}  in the following equivalent form: for every $y^1_1,\ldots, y^1_n \in S_1$ and $y^2_1,\ldots, y^2_n \in S_2$ it holds
\be\lab{eq:con-ind-fidis-2}
	\P\left( \bigcap_{k=1}^n (X^1_{t_k}, X^2_{t_k}) =(y^1_k,y^2_k) | \F_T \right)
	= \P \left(\bigcap_{k=1}^n \set{X^1_{t_k}=y^1_k }| \F_T \right) \P \left(\bigcap_{k=1}^n \set{X^2_{t_k}=y^2_k} | \F_T \right).
\ee
 By the tower property of conditional expectations, by definition of $(\FF,\GG)$-DSMC, by \cite[Proposition 4.6]{BieJakNie2014a}, and by \eqref{eq:N-in-dist-cond-ind}  we can rewrite the left hand side of \eqref{eq:con-ind-fidis-2} as follows
\allowdisplaybreaks
\begin{align*}
&	\P \Big((X^1_{t_1}, X^2_{t_1}) = (y^1_1,y^2_1 ) ,\ldots,
(X^1_{t_n}, X^2_{t_n}) = (y^1_n,y^2_n)
 | \F_T \Big)
\\
&=
\E \Big(	\P(X^1_{t_1}, X^2_{t_1}) = (y^1_1,y^2_1 ) ,\ldots
(X^1_{t_n}, X^2_{t_n}) = (y^1_n,y^2_n)
 | \F_T \vee \G_0) | \F_T \Big)\\
&=\E\bigg(	\sum_{(y^1_0,y^2_0) \in S_1\times S_2}
\I_\set{X^1_0 =y^1_0, X^2_0 = y^2_0} \prod_{k=1}^n p_{(y^1_{k-1},y^2_{k-1})(y^1_k, y^2_k)}(t_{k-1},t_k) | \F_T \bigg)
\\
&=\sum_{(y^1_0,y^2_0) \in S_1\times S_2} \P\left(	{X^1_0 =y^1_0, X^2_0 = y^2_0} | \F_T \right)  \prod_{k=1}^n p_{(y^1_{k-1},y^2_{k-1})(y^1_k, y^2_k)}(t_{k-1},t_k)
\\
&=\sum_{y^1_0 \in S_1 }
\sum_{y^2_0 \in S_2 }\P\left(	{X^1_0 =y^1_0}| \F_T \right)   \P\left( {X^2_0 = y^2_0} | \F_T \right)
 \prod_{k=1}^n p_{1;y^1_{k-1}y^1_k}(t_{k-1},t_k)
p_{2;y^2_{k-1}, y^2_k}(t_{k-1},t_k)
\\
&=
\bigg(
\sum_{y^1_0 \in S_1 }
\P\left(	{X^1_0 =y^1_0}| \F_T \right)
\ \prod_{k=1}^n p_{1;y^1_{k-1}y^1_k}(t_{k-1},t_k)
\bigg)
%\\
%& \quad
\bigg(
\sum_{y^2_0 \in S_2 }
\P\left( {X^2_0 = y^2_0} | \F_T \right)
 \prod_{k=1}^n
p_{2;y^2_{k-1}, y^2_k}(t_{k-1},t_k)
\bigg). \end{align*}
\begin{center}

\end{center}
Summing the above equality over all $y^2_1 \ldots, y^2_n \in S_2$ yields
\[
\P(X^1_{t_1} =y^1_1, \ldots, X^1_{t_n} =y^1_n | \F_T ) =
\sum_{y^1_0 \in S_1 }
\P\left(	{X^1_0 =y^1_0}| \F_T \right)
 \prod_{k=1}^n P^1_{y^1_{k-1}y^1_k}(t_{k-1},t_k).
\]
In analogous way we obtain that
\[
\P(X^2_{t_1} =y^2_1, \ldots, X^2_{t_n} =y^2_n | \F_T )=
\sum_{y^2_0 \in S_2 }
\P\left( {X^2_0 = y^2_0} | \F_T \right)
 \prod_{k=1}^n
P^2_{y^2_{k-1}, y^2_k}(t_{k-1},t_k) .
\]
This ends the proof of \eqref{eq:con-ind-fidis-2}.
\finproof
\end{proof}
\begin{proposition}\label{cor:cond-ind}
Let $\Lambda$ be given by \eqref{eq:LLL}. Suppose  that \eqref{eq:N-in-dist-cond-ind} is satisfied. Then, the components $X^1,\ldots, X^N$ of  $X$ are conditionally independent given $\F_T$.
\end{proposition}
\begin{proof}
The assertion follows immediately from Lemma \ref{lem:811} and Proposition \ref{prop:cond-indep-DSMC}  .
\end{proof}

\begin{corollary}
Let $X=(X^1, \ldots , X^N)$ be an $(\FF, \FF^X)$-DSMC. Suppose that \eqref{eq:Ndim-Kron-prod} and \eqref{eq:N-in-dist-cond-ind} hold. If additionally
\be\label{eq:cosik-init}
\P(X^k_0 =x^k | \F_T ) =\P(X^k_0 =x^k | \F_0 ), \quad \forall x^k \in S_k, \ k=1, \ldots, N,
\ee
then the components of $X$ are conditionally independent $(\FF, \FF^X)$-CMCs.
\end{corollary}
\begin{proof}
The assumption \eqref{eq:cosik-init} and \eqref{eq:N-in-dist-cond-ind} imply by \cite[Corollary 4.7]{BieJakNie2014a}
that $\FF$ is $\P$-immersed in $\FF \vee \FF^X$. Thus using \cite[Proposition 4.13]{BieJakNie2014a} we conclude that $X$ is $(\FF,\FF^X)$-CMC. The conditional independence of components of $X$ follows from Proposition \ref{prop:cond-indep-DSMC}.
\end{proof}

\subsection*{Acknowledgments}
Research of T.R. Bielecki was partially supported by NSF grant DMS-1211256.


\begin{thebibliography}{10}

\bibitem{BalYeo1993}
F.~Ball and G.~F. Yeo.
\newblock Lumpability and marginalisability for continuous-time {M}arkov
  chains.
\newblock {\em J. Appl. Probab.}, 30(3):518--528, 1993.

\bibitem{BiaWid2013}
F.~Biagini, A.~Groll, and J.~Widenmann.
\newblock Intensity-based premium evaluation for unemployment insurance
  products.
\newblock {\em Insurance Math. Econom.}, 53(1):302--316, 2013.

\bibitem{BCCH2014}
T.~R. Bielecki, A.~Cousin, S.~Crepey, and A.~Herbertsson.
\newblock Dynamic hedging of portfolio credit risk in a {M}arkov copula model.
\newblock {\em Journal of Optimization Theory and Applications}, 161(1), 2014.

\bibitem{BJN-buczek}
T.~R. Bielecki, J.~Jakubowski, and M.~Niew\k{e}g{\l}owski.
\newblock Fundamentals of {T}heory of {S}tructured {D}ependence between
  {S}tochastic {P}rocesses.
\newblock Book in preparation.

\bibitem{BieJakNie2010}
T.~R. Bielecki, J.~Jakubowski, and M.~Niew\k{e}g{\l}owski.
\newblock Dynamic modeling of dependence in finance via copulae between
  stochastic processes.
\newblock In {\em Copula theory and its applications}, volume 198 of {\em Lect.
  Notes Stat. Proc.}, pages 33--76. Springer, Heidelberg, 2010.

\bibitem{BieJakNie2012}
T.~R. Bielecki, J.~Jakubowski, and M.~Niew\k{e}g{\l}owski.
\newblock Study of dependence for some stochastic processes: symbolic {M}arkov
  copulae.
\newblock {\em Stochastic Process. Appl.}, 122(3):930--951, 2012.

\bibitem{BieJakNie2013}
T.~R. Bielecki, J.~Jakubowski, and M.~Niew\k{e}g{\l}owski.
\newblock Intricacies of dependence between components of multivariate {M}arkov
  chains: weak {M}arkov consistency and weak {M}arkov copulae.
\newblock {\em Electron. J. Probab.}, 18:no. 45, 21, 2013.

\bibitem{BieJakNie2015}
T.~R. {Bielecki}, J.~{Jakubowski}, and M.~{Niew\k{e}g{\l}owski}.
\newblock {Conditional Markov chains -- construction and properties.}
\newblock In {\em Banach Center Publications 105(2015) Stochastic analysis.
  Special volume in honour of Jerzy Zabczyk.}, pages 33--42. Warsaw: Polish
  Academy of Sciences, Institute of Mathematics, 2015.

\bibitem{BieJakNie2014a}
T.~R. Bielecki, J.~Jakubowski, and M.~Niew\k{e}g{\l}owski.
\newblock Conditional {M}arkov chains, {p}art {I}: {c}onstruction and
  properties.
\newblock
  \href{http://arxiv.org/abs/1501.05531}{(http://arxiv.org/abs/1501.05531)},
  2015.

\bibitem{BieJakVidVid2008}
T.~R. Bielecki, J.~Jakubowski, A.~Vidozzi, and L.~Vidozzi.
\newblock Study of dependence for some stochastic processes.
\newblock {\em Stoch. Anal. Appl.}, 26(4):903--924, 2008.

\bibitem{BieRut2001}
T.~R. Bielecki and M.~Rutkowski.
\newblock {\em {Credit risk: Modelling, valuation and hedging.}}
\newblock {Springer Finance. Berlin: Springer. xviii, 500~p.}, 2002.

\bibitem{BieVidVid2008}
T.~R. Bielecki, A.~Vidozzi, and L.~Vidozzi.
\newblock Markov copulae approach to pricing and hedging of credit index
  derivatives and ratings triggered step - up bonds.
\newblock {\em Journal of Credit Risk}, 4(1):47--76, 2008.

\bibitem{BurRos1958}
C.~J. Burke and M.~Rosenblatt.
\newblock A {M}arkovian function of a {M}arkov chain.
\newblock {\em Ann. Math. Statist.}, 29:1112--1122, 1958.

\bibitem{ethkur1986}
S.~N. Ethier and T.~G. Kurtz.
\newblock {\em Markov processes}.
\newblock Wiley Series in Probability and Mathematical Statistics: Probability
  and Mathematical Statistics. John Wiley \& Sons Inc., New York, 1986.
\newblock Characterization and convergence.

\bibitem{FolProt2011}
H.~F{\"o}llmer and P.~Protter.
\newblock Local martingales and filtration shrinkage.
\newblock {\em ESAIM Probab. Stat.}, 15(In honor of Marc Yor, suppl.):S25--S38,
  2011.

\bibitem{Granger}
C.~W.~J. Granger.
\newblock Investigating causal relations by econometric models and
  cross-spectral methods.
\newblock {\em Econometrica}, 1969.

\bibitem{HeWanYan1992}
S.~W. He, J.~G. Wang, and J.~A. Yan.
\newblock {\em Semimartingale {T}heory and {S}tochastic {C}alculus}.
\newblock Kexue Chubanshe (Science Press), Beijing, 1992.

\bibitem{HorJoh1994}
R.~A. Horn and C.~R. Johnson.
\newblock {\em Topics in matrix analysis}.
\newblock Cambridge University Press, Cambridge, 1994.
\newblock Corrected reprint of the 1991 original.

\bibitem{JakNie2007}
J.~Jakubowski and M.~Niew\k{e}g\l{}owski.
\newblock {Pricing bonds and CDS in the model with rating migration induced by
  a Cox process.}
\newblock {Stettner, \L ukasz (ed.), Advances in mathematics of finance. Banach
  Center Publications 83, 159-182.}, 2008.

\bibitem{JakNie2010}
J.~Jakubowski and M.~Niew\k{e}g{\l}owski.
\newblock A class of {$\Bbb F$}-doubly stochastic {M}arkov chains.
\newblock {\em Electron. J. Probab.}, 15:no. 56, 1743--1771, 2010.

\bibitem{JakPyt2015}
J.~Jakubowski and A.~Pytel.
\newblock The {M}arkov consistency of {A}rchimedean survival processes.
\newblock {\em Forthcoming in Journal of Applied Probability}, 2015.

\bibitem{LiangDong}
X.~Liang and Y.~Dong.
\newblock A {M}arkov chain copula model for credit default swaps with bilateral
  counterparty risk.
\newblock {\em Comm. Statist. Theory Methods}, 43(3):498--514, 2014.

\end{thebibliography}
\end{document}